\documentclass[11pt,reqno,a4paper]{amsart}
\usepackage[latin1]{inputenc}
\usepackage[english]{babel}
\usepackage[T1]{fontenc}
\usepackage{amsmath,amsthm,amssymb, color,graphicx,bm,mathtools}
\usepackage{amsmath} 
\usepackage{amsfonts}
\usepackage{amssymb}
\usepackage{amsthm}
\usepackage{dsfont}
\usepackage{stmaryrd}
\usepackage{tikz}
\usepackage{todonotes}
\usetikzlibrary{matrix,arrows.meta}
\usepackage[title,titletoc,toc]{appendix}
\usepackage[colorlinks=true,linkcolor=blue]{hyperref}
\providecommand{\doi}[1]{{\hypersetup{urlcolor=black}\href{#1}{#1}}}
\usepackage[normalem]{ulem}

\allowdisplaybreaks

\usepackage{xcolor}

\usepackage[a4paper,top=2.4cm,bottom=3cm,left=2.5cm,right=2.4cm]{geometry}

\theoremstyle{plain}
\newtheorem{theorem}{Theorem}[section]
\newtheorem{lemma}[theorem]{Lemma}

\theoremstyle{definition} 
\newtheorem{remark}[theorem]{Remark}
\newtheorem{definition}[theorem]{Definition}

\newcommand{\R}{\mathbb{R}}

\renewcommand{\d}[1]{\ensuremath{\operatorname{d}\!{#1}}}

\begin{document}

\author{Luca Di Persio}
\address[L.\,Di Persio]{Dipartimento di Informatica, Universit\`a degli Studi di Verona, Strada le Grazie 15, 37134 Verona, Italy}
\email{luca.dipersio@univr.it}
\author{Riccardo Molinarolo}
\address[R.\,Molinarolo]{Dipartimento di Informatica, Universit\`a degli Studi di Verona, Strada le Grazie 15, 37134 Verona, Italy}
\email{riccardo.molinarolo@univr.it}

\title[]{Shape analysis in Schauder spaces of the energy of heat problems in perturbed annular domains}

\date{\today}

\maketitle

\noindent
{\textbf{Abstract:}} This paper is devoted to the shape analysis of the energy of a caloric family of Schauder functions defined on a bounded perforated domain $\Omega^o \setminus \overline{\Omega^i[\phi]}$ of $\R^n$, where the outer boundary is fixed, and the inner boundary is obtained by a $C^{1,\alpha}$-perturbation $\phi$ of the boundary of a reference cavity $\Omega^i$. Without imposing any boundary conditions, we prove that in a suitable neighborhood of the identity $\phi_0$, the domain-to-energy map is of class $C^{\infty}$. The proof is based on the construction of a global diffeomorphism, smoothly depending on $\phi$, from the reference annulus onto the perturbed one and on suitable regularity and smoothness assumptions on the pull-back family onto the reference domain.

We then apply our main result to two boundary value problems: a nonlinear mixed Robin-type problem and a linear Dirichlet problem. After recalling some known existence and shape analysis results for the solutions, we prove that the corresponding domain-to-energy map is of class $C^{\infty}$. The proof is based on a decomposition of the fixed domain into near, intermediate, and far regions relative to the cavity, and on the smooth dependence of the layer heat potentials upon support perturbations.

\vspace{9pt}

\vspace{9pt}


\tableofcontents

\section{Introduction} 
The present work lies at the intersection of parabolic potential theory, shape sensitivity analysis, and boundary integral methods for time-dependent problems. For the general theory of parabolic equations and classical Schauder regularity, we refer the reader to the classical monographs by Lady\v{z}enskaja, Solonnikov, and Ural'tseva \cite{LaSoUr68}, Friedman \cite{Friedman2008}, Lieberman \cite{Li96}, and Lunardi \cite{Lu95}.

The paper aims to study the dependence of the energy of caloric Schauder functions in bounded perforated domains upon perturbation of the internal boundary. 

Generally speaking, understanding how certain quantities depend on the shape of the domain is a fundamental quest with significant relevance in many real-world scenarios, especially when one searches for optimal shapes or configurations. In mathematical terms, this leads to the well-established field of \textit{Shape Optimization}, for which we refer the reader to the classical monographs by Henrot and Pierre \cite{HePi18} and by Soko{\l}owski and Zolesio \cite{SoZo92}, along with the references therein.

From a mathematical perspective, the analysis of the shape optimization problem is often closely related to studying how the solutions - or other related quantities - to some boundary value problems are affected by the perturbation of the domain of definition. In particular, this requires the understanding of the regularity of the map that associates the shape of the domain to the solution, or to other related quantities, of the problems.
Continuity gives stability under small perturbations, differentiability yields first-order optimality conditions, and higher regularity provides higher-order expansions and a robust functional analytic framework for perturbative arguments.

Many authors have used potential theoretic methods to study the perturbation sensitivity of solutions of boundary value problems, exploiting integral operators to transform the original problem into a system of boundary integral equations. We mention the classical monograph on integral equations by Kress \cite{Kr14}, as well as the systematic analysis of boundary integral operators for the heat equation by Brown \cite{Br89,Br90} and by Costabel \cite{Co90,Co04}. In particular, analyzing how layer potentials and, more generally, integral operators depend on perturbation parameters (such as densities or the support of integration) becomes a crucial first step.

This strategy has been performed for the sensitivity analysis of the solutions of boundary value problems, especially in the context of elliptic equations. We mention Potthast \cite{Pot94,Pot96} for the study of layer potential for the Helmholtz equation, Haddar and Kress \cite{HaKr04}, Kirsch \cite{Kir93}, Charalabopoulos \cite{Chara95}, Costabel and Le Louer \cite{CoLeLo12I,CoLeLo12II} for applications to scattering theory, fluid dynamics, and elasticity theory.  

Few results prove regularity beyond differentiability. Notably, exceptions are the works of Lanza de Cristoforis and collaborators. The method developed was called \textit{Functional Analytic Approach}, and it heavily relies on potential-theoretic techniques: for a detailed description of the method, see the monograph \cite{DaLaMuBook} and reference therein. This method has been used for both regular and singular perturbations and, generally speaking, it aims to prove smooth or even analytical results for the layer potential and for the domain-to-solution map. It has been extensively adopted: far from being exhaustive, we refer to the seminal papers \cite{La07, LaRo04,LaRo08} for the smoothness of classical layer potentials, and the more general case of elliptic differential operators studied in \cite{DaLa10}.

For the sake of completeness, we mention that, apart from the results stemming from Lanza de Cristoforis's work, it is possible to obtain analyticity results through \textit{Shape Holomorphy}. We refer to Henr\'iquez and Schwab \cite{HeSc21} for the study of the Calder\'on projector, to Pinto, Henr\'iquez, and Jerez-Hanckes \cite{PiHeJe24} and D\"olz and Henr\'iquez \cite{DoHe24} for boundary integral operators, and to Jerez-Hanckes, Schwab, and Zech \cite{JeScZe17} for the electromagnetic wave scattering problem.

As regards shape analysis for parabolic problems, we mention the theoretical results for parabolic shape optimization problems with time-independent shapes obtained by Soko{\l}owski \cite{So1988} (see also \cite{SoZo92, ElSo1996}). In particular, nowadays, shape calculus and shape optimization for parabolic problems are already well developed: in the context of inverse and identification problems for the
heat equation, relevant contributions include Bacchelli,
Di Cristo, Sincich and Vessella \cite{BaDiSiVe14}, Chapko, Kress, and Yoon \cite{ChKrYoo98,ChKrYoo99}, Hettlich and Rungell \cite{HeRu01}, and Nakamura and Wang \cite{NaWa15,NaWa17}. We also mention the monograph by Isakov \cite{Is98} on the topic.

The work of Harbrecht and Tausch \cite{HaTa2013} is especially relevant. They considered the numerical solution of a shape identification problem for the heat equation with the aim to determine the shape of a void or an inclusion of zero temperature from measurements of the temperature and the heat flux at the exterior boundary through different shape optimization problems (see also \cite{HaTa2011}).
The subsequent works of Br\"ugger, Harbrecht and Tausch \cite{BrHaTa2021,BrHaTa2022} address an additional level of complexity (that is not present in our paper), namely time-dependent shapes. In \cite{BrHaTa2021}, a time-dependent shape identification problem for the heat equation is formulated and treated by a gradient-based method. In \cite{BrHaTa2022}, the authors develop a functional analytic Sobolev framework for heat boundary integral equations on non-cylindrical space-time domains. 

More recently, Djurdjevac, Kaarnioja, Schillings and Zepernick \cite{DjKaScZe2026} study uncertainty quantification for stationary and time-dependent PDEs under Gevrey regular random domain deformations. In particular, the model adopted to describe the random domain follows the approach of Harbrecht, Schmidlin, and Schwab \cite{HaScSc2024}. 

Finally, a very recent work by Harbrecht and Schwab \cite{HaSc2026} provides error bounds for operator surrogates of solution operators for partial differential and boundary integral equations on families of
domains which are diffeomorphic to one common reference domain. Furthermore, they obtain constructive proofs of the existence of neural and spectral operator surrogates for the shape-to-solution maps with error bounds and convergence rate guarantees uniform on the collection of admissible shapes.

Within this broad research framework, the present paper provides an abstract result on the smoothness of the energy of a family of caloric functions on perforated domains under static and deterministic shape perturbations of the cavity. In particular, in our main result, we deal with classical space-time Schauder functions rather than Sobolev functions. 

There are several reasons for this choice. First, our work contributes to and further develops an established line of research. The closest antecedents of the present paper are the recent works \cite{DaLuMoMu24.2,DaLuMu25}, where shape analysis was carried out for a nonlinear mixed Robin boundary value problem for the heat equation and for the Dirichlet-to-Neumann map, respectively, in perturbed perforated domains. Related ideas were also developed for periodic transmission problems in \cite{Lu20,DaLuMoMu24}. All these results rely crucially on the sharp regularity results for the pullbacks of layer heat potentials established by Luzzini and Dalla Riva \cite{DaLu23} (see also \cite{LaLu17,LaLu19}). In all of these works, the analysis is carried out within the framework of Schauder spaces.

The novelty of the present manuscript is the shape analysis of the energy associated with a family of caloric functions in perturbed perforated domains. As we will describe below, in the second part of the paper, we propose some applications of our result to the analysis of the corresponding energy of solutions of two parabolic boundary value problems. For elliptic boundary value problems and for singular perturbation, the asymptotic behavior of the energy of solutions of a two-parameter homogenization problem was analyzed by Lanza de Cristoforis and Musolino \cite{LaMu19}. To the best of our knowledge, the shape analysis of the energy of perturbed parabolic problems by means of potential-theoretic techniques and in Schauder spaces is new.

\subsection{Main result}
In order to present our main result, we describe the geometric framework by fixing once and for all 
\begin{equation*}
    n \in \mathbb{N} \setminus \{0,1\},
\end{equation*}
which denotes the dimension of the Euclidean space $\mathbb{R}^n$. Then, we fix a regularity parameter and a final time, namely
\begin{equation*}
    \alpha \in ]0,1[ \text{ and } T>0,
\end{equation*}
and we consider two sets $\Omega^o$ and $\Omega^i$ that satisfy the following conditions:
\begin{equation}\label{introsetconditions}
    \begin{split}
        &\Omega^o, \Omega^i  \text{ are bounded open connected subsets of } \R^n \text{ of class } C^{1,\alpha}, 
        \\
        &\text{with connected exteriors } \R^n\setminus \overline{\Omega^o} \text{ and } \R^n \setminus \overline{\Omega^i} \text{ and } \overline{\Omega^i} \subseteq\Omega^o.
    \end{split}
\end{equation}

Then, we define the perforated domain obtained by removing from the fixed set $\Omega^o$ a perturbed copy of the set $\Omega^i$. More precisely, we will perturb the domain $\Omega^i$ by means of a diffeomorphism $\phi$ in the class
\begin{equation*}
    \mathcal{A}_{\partial\Omega^i} := \left\{ \phi \in C^{1,\alpha}(\partial\Omega^i,\R^n) : \phi \text{ injective}, \, \d \phi(y) \text{ injective for all } y \in \partial\Omega^i \right\}.
\end{equation*}

Now, it is well-known that $\phi(\partial\Omega^i)$ splits $\R^n$ into exactly two open connected components, one bounded, which we denote by $\Omega^i[\phi]$, and one unbounded, $\Omega^i[\phi]^- := \R^n \setminus \Omega^i[\phi]$ (see the explanation at the beginning of Subsection \ref{subsec: class of diffeo}). Moreover, in order to generate a proper perturbed perforated domain, we introduce the subclass
\begin{equation*}
    \mathcal{A}^{\Omega^o}_{\partial\Omega^i} := \left\{ \phi \in \mathcal{A}_{\partial\Omega^i} \colon  \overline{\Omega^i[\phi]} \subseteq \Omega^o \right\}.
\end{equation*}

Then, for $\phi \in \mathcal{A}^{\Omega^o}_{\partial\Omega^i}$, we consider the perturbed perforated domain $\Omega^o \setminus \overline{\Omega^i[\phi]}$ (see Figure \ref{fig: perturbed domains}). We observe that, by assumption \eqref{introsetconditions}, $\phi_0 := \mathrm{id}_{\partial\Omega^i} \in  \mathcal{A}^{\Omega^o}_{\partial\Omega^i}$ and $\Omega^o \setminus \overline{\Omega^i[\phi_0]} = \Omega^o \setminus \overline{\Omega^i}$.

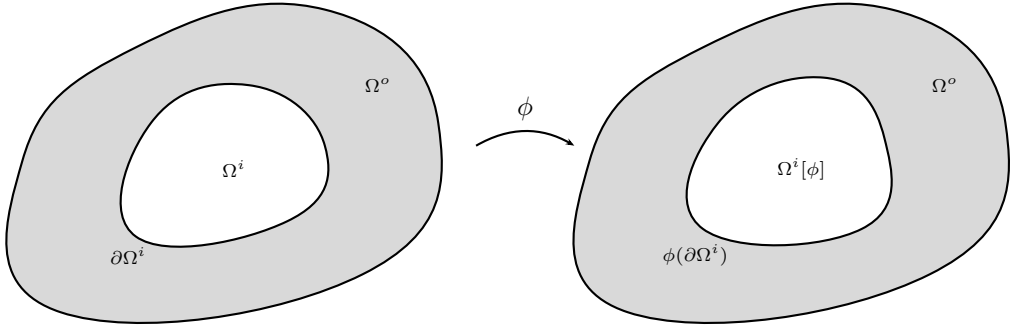
\begin{figure}[ht!]
    \centering
    \begin{tikzpicture}

    \fill[gray!30]
      plot[smooth cycle, tension=0.8]
      coordinates {(0,0) (4,0.5) (5,2.5) (3.5,4) (1,3.5) (-0.5,2)};
    \draw[thick, black]
      plot[smooth cycle, tension=0.8]
      coordinates {(0,0) (4,0.5) (5,2.5) (3.5,4) (1,3.5) (-0.5,2)};
    
    \def\inner{
      (1,1) (3,1.2) (3.5,2.2) (2.5,3) (1.2,2.5)
    }
    \fill[white] plot[smooth cycle, tension=0.8] coordinates \inner;
    \draw[thick, black] plot[smooth cycle, tension=0.8] coordinates \inner;
    \node at (4.2,3) {\tiny$\Omega^o$};
    \node at (2.3,1.9) {\tiny$\Omega^i$};
    \node at (0.9,0.75) {\tiny$\partial\Omega^i$};
    \draw[->, thick, >={Stealth[scale=0.5]}] (5.5,2.2) to[bend left=30] node[midway, above] {\small$\phi$} (6.8,2.2);
    \begin{scope}[xshift=7.5cm]
        \fill[gray!30] plot[smooth cycle, tension=0.8] coordinates {(0,0) (4,0.5) (5,2.5) (3.5,4) (1,3.5) (-0.5,2)};
        \draw[thick, black] plot[smooth cycle, tension=0.8] coordinates {(0,0) (4,0.5) (5,2.5) (3.5,4) (1,3.5) (-0.5,2)};
        \def\inner{(1,1.1) (3.1,1.1) (3.4,2.3) (2.6,3.1) (1.2,2.5)}
        \fill[white] plot[smooth cycle, tension=0.8] coordinates \inner;
        \draw[thick, black] plot[smooth cycle, tension=0.8] coordinates \inner;
        \node at (4.2,3) {\tiny$\Omega^o$};
        \node at (2.3,1.9) {\tiny$\Omega^i[\phi]$};
        \node at (0.9,0.75) {\tiny$\phi(\partial\Omega^i)$};
    \end{scope}
    \end{tikzpicture}
    \caption{The domains $\Omega^o$, $\Omega^i$, the diffeomorphism $\phi$, the sets $\phi(\partial \Omega^i)$ and $\Omega^i[\phi]$.}
    \label{fig: perturbed domains}
\end{figure}

Lanza de Cristoforis and Rossi \cite{LaRo08} proved that every diffeomorphism $\phi \in \mathcal{A}^{\Omega^o}_{\partial\Omega^i}$ in a neighbourhood of $\phi_0$ can be extended in a suitably defined uniform neighborhood of $\partial\Omega^i$ by means of a real-analytic extension operator $\mathbf{E}$. Building upon that result, as a first step, we prove the existence of an open neighborhood $\mathcal{W}^\ast_0$ of $\phi_0$ in $\mathcal{A}^{\Omega^o}_{\partial\Omega^i}$ and of a map $\Psi$ such that, for every $\phi \in \mathcal{W}^\ast_0$, $\Psi[\phi]$ is a regular global diffeomorphism from $\overline{\Omega^o}\setminus\Omega^i$ onto $\overline{\Omega^o} \setminus \Omega^i[\phi]$, which coincides with the identity map in a neighborhood of $\partial\Omega^o$ and with $\mathbf{E}[\phi]$ in a neighborhood of $\partial\Omega^i$. More importantly, $\Psi[\phi]$ depends real analytically on $\phi$. We refer to Lemma \ref{lemma diffeo} in Section \ref{sec: abstract result} for the precise and technical statement of that result. With this notation in hand, we can state our main result.

\begin{theorem}\label{thm energy dependence abstract}
    Assume that there exist an open neighborhood $\widehat{\mathcal{W}}_0 \subseteq \mathcal{W}^\ast_0$ of $\phi_0$ in $\mathcal{A}^{\Omega^o}_{\partial\Omega^i}$ and a family of functions $\{ u_\phi \}_{\phi \in \widehat{\mathcal{W}}_0}$ such that the following holds:
    \begin{itemize}
        \item[(i)] for all $\phi \in \widehat{\mathcal{W}}_0$, $u_\phi \in C_{0}^{\frac{1+\alpha}{2}; 1+\alpha}([0,T] \times (\overline{\Omega^o} \setminus \Omega^i[\phi] ))$ and satisfy
        \begin{equation}\label{princeqpertu family}
            \partial_t u - \Delta u = 0  \quad \text{in } ]0,T] \times (\Omega^o \setminus \overline{\Omega^i[\phi]});
        \end{equation}

        \item[(ii)] the map from $\widehat{\mathcal{W}}_0$ to $C_{0}^{\frac{1+\alpha}{2}; 1+\alpha}([0,T] \times (\overline{\Omega^o} \setminus \Omega^i))$, which takes $\phi$ to the function $U_\phi$ given by 
        \begin{equation}\label{eq: U_phi := u_phi Psi phi}
        U_\phi(t,x):= \left( u_\phi \circ \Psi[\phi]^T \right) (t,x) \quad \text{for } (t,x) \in \overline{\Omega^o} \setminus \Omega^i,
        \end{equation}
        where $\Psi[\phi]^T(t,x):=(t,\Psi[\phi](x))$, is of class $C^\infty$.
    \end{itemize}
    Then, the map from $\widehat{\mathcal{W}}_0$ to $C^1([0,T])$, which takes $\phi$ to the function $e_\phi$ given by
    \begin{equation*}
    e_\phi(t) := \frac{1}{2} \int_{\Omega^o \setminus \overline{\Omega^i[\phi]}} (u_\phi(t,y))^2 \d y \quad \forall t \in [0,T],
    \end{equation*}
    is of class $C^\infty$. In particular, for every $t\in(0,T)$,
    \begin{equation}\label{eq: derivative of e general}
    \begin{aligned}
        \frac{\d{}}{\d t}e_\phi(t) &= - \int_{\Omega^o \setminus \overline{\Omega^i[\phi]}} |\nabla u_\phi(t,y)|^2 \d y + \int_{\partial\Omega^o} 
        \nabla u_\phi(t,y)\cdot \nu_{\Omega^o}(y) \, u_{\phi}(t,y) \d \sigma_y 
        \\
        & \quad
        - \int_{\partial\Omega^i} \nabla u_\phi(t,\phi(y))\cdot \nu_{\Omega^i[\phi]}(\phi(y)) \, u_\phi(t,\phi(y)) \, \tilde\sigma_n[\phi](y) \d\sigma_y,
    \end{aligned}
    \end{equation}
    and the right-hand side extends continuously to $[0,T]$.
\end{theorem}

Our main result provides sufficient conditions for the smoothness of the domain-to-energy map for a family of Schauder functions. In particular, the novelty of the paper lies in the construction of the global diffeomorphism $\Psi[\phi]$ which maps the reference fixed annulus onto the perturbed
annulus, depending real analytically on the perturbation parameter $\phi$ (see Lemma \ref{lemma diffeo}), and using it to pull back the energy to the fixed domain. Furthermore, for the sake of clarity, in the proof of Theorem \ref{thm energy dependence abstract}, we first prove the $C^\infty$ dependence upon the perturbation $\phi$ of the energy $e_\phi$ as a $C^0([0,T])$-valued map. We then prove the $C^\infty$ dependence of its time derivative as a $C^0([0,T])$-valued map by using the pulled-back energy identity. This yields the stated smoothness of the energy as a map with values in $C^1([0,T])$. Both steps make precise use of the assumptions on the family $\{ u_\phi \}_{\phi \in \widehat{\mathcal{W}}_0}$ and on the pull-back family $\{ U_\phi \}_{\phi \in \widehat{\mathcal{W}}_0}$, see assumptions (i)-(ii) in Theorem \ref{thm energy dependence abstract}.

In the final part of the paper, we present two applications of our main result to boundary value problems: a nonlinear mixed Robin-type problem and a linear Dirichlet problem. We first briefly recall the existence and smoothness of the domain-to-solution map established in \cite{DaLuMoMu24.2} and \cite{DaLuMu25}, respectively, using the \textit{Functional Analytic Approach} and potential-theoretic techniques. We then apply Theorem \ref{thm energy dependence abstract} to prove the smoothness of the corresponding domain-to-energy map. The proof relies on the smooth dependence, with respect to perturbations of both the support and the density, of the $\phi$-pullbacks of the layer heat potentials established by Dalla Riva and Luzzini in \cite{DaLu23}. These fundamental results are recalled in Theorems \ref{thm dependence V Phi W Phi} and \ref{thm dependence V phi W phi} and are central to prove that the family of solutions provided by \cite{DaLuMoMu24.2} and \cite{DaLuMu25} satisfy assumption (ii) of Theorem \ref{thm energy dependence abstract}.

\subsection{Structure of the paper} The paper is organized as follows. In Section \ref{s:prel} we fix some notations, we recall the definition of parabolic Schauder spaces, the properties of the single layer heat potential, the relative boundary integral operators, the class of diffeomorphisms used to model the shape of the cavity and the smooth dependence of pullback layer heat potentials on perturbations of the support. Section \ref{sec: abstract result} contains the novelty of the paper and the proof of the main results: Lemma \ref{lemma diffeo} on the construction of the global diffeomorphism depending real analytically on the perturbation parameter and the proof of Theorem \ref{thm energy dependence abstract} via the pull-back of the energy to the fixed domain. Finally, in Section \ref{sec: application} we present two applications of our abstract result to boundary value problems, a nonlinear mixed Robin-type problem and a linear Dirichlet problem, proving that the corresponding domain-to-energy map associated to suitable family of solutions of those problems is of class $C^\infty$. 

\section{Preliminaries}\label{s:prel}
\subsection{Notation and parabolic Schauder spaces}
The symbol $\mathbb{N}$ denotes the set of natural numbers that includes $0$. Throughout the paper,
\[
n \in \mathbb{N} \setminus \{0, 1\}
\]
denotes the dimension of the Euclidean ambient space $\mathbb{R}^n$. The ball of center $x \in \mathbb{R}^n$ and radius $r\geq 0$ is denoted by $B(x,r)$. The inverse of an
invertible function $f$ is denoted by $f^{(-1)}$, while the reciprocal of function $g$ is denoted by $g^{-1}$. The transpose of a matrix $A \in \R^{n\times n}$ is denoted by $A^\top$, and the identity matrix by $I_n$.

If $\Omega$ is an open subset of $\mathbb{R}^n$, then $\overline{\Omega}$ denotes
the closure of $\Omega$, $\partial \Omega$ denotes the boundary of $\Omega$ and $\Omega^- := \mathbb{R}^n \setminus \overline{\Omega}$ denotes the exterior of $\Omega$. 

Let $m\in \mathbb{N}$ and $\alpha \in ]0,1[$. For the definition of open subsets of $\mathbb{R}^n$ of class $C^m$
and $C^{m,\alpha}$, and of the Schauder spaces $C^{m,\alpha}(\overline{\Omega})$ and $C^{m,\alpha}(\partial \Omega)$, we refer to Gilbarg and
Trudinger \cite[p.\,52\,\&\,p.\,94]{GiTr83}. In particular, we use the subscript ``$b$'' to denote the subspace consisting of those functions that are bounded.

For a comprehensive introduction to parabolic Schauder spaces, we refer the reader to classical monographs on the field, for example, Lady\v{z}enskaja, Solonnikov, and Ural'ceva
\cite[Ch.\,1]{LaSoUr68}. For completeness and the reader's convenience, we recap some useful definitions.

Let $\alpha \in ]0,1[$, $T>0$, and let $\Omega$ be an open subset of $\mathbb{R}^n$. Then we set 
\begin{equation*}
    C^{\frac{\alpha}{2},\alpha}([0, T] \times \overline{\Omega}) := \Big\{ u \in C^{0}_{b}([0, T] \times \overline{\Omega}) \, \colon \|u\|_{C^{\frac{\alpha}{2};\alpha}([0, T] \times \overline{\Omega})} < +\infty \Big\},
\end{equation*}
where
\begin{align*}
    \|u\|_{C^{\frac{\alpha}{2};\alpha}([0, T] \times \overline{\Omega})} :=& \sup_{[0, T] \times \overline{\Omega}} |u| + \sup_{\substack{t_1,t_2 \in [0,T] \\ t_1 \neq t_2}} \,\sup_{x \in \overline{\Omega}} \frac{|u(t_1,x) - u(t_2,x)|}{|t_1-t_2|^\frac{\alpha}{2}}
    \\
    & + \sup_{t \in [0,T]} \,\sup_{\substack{x_1,x_2 \in \overline{\Omega} \\ x_1 \neq x_2}} \frac{|u(t,x_1) - u(t,x_2)|}{|x_1-x_2|^\alpha} .
\end{align*}
Then we set
\begin{align*}
    C^{\frac{1+\alpha}{2}; 1+\alpha}([0, T] \times \overline{\Omega}) := \Big\{ u \in C^{0}_{b}([0, T] \times \overline{\Omega}) \, \colon \partial_{x_i} u \in C^{0}_b([0, T] \times \overline{\Omega}) \text{ for all } i \in \{1,\dots,n\}, 
    \\
    \text{and such that }
    \|u\|_{C^{\frac{1+\alpha}{2}; 1+\alpha}([0, T] \times \overline{\Omega})}
    < +\infty \Big\},
\end{align*}
where
\begin{align*}
    \|u\|_{C^{\frac{1+\alpha}{2}; 1+\alpha}([0, T] \times \overline{\Omega})} := \sup_{[0, T] \times \overline{\Omega}} |u| + \sum_{i=1}^{n} \|\partial_{x_i} u\|_{C^{\frac{\alpha}{2}; \alpha}([0, T] \times \overline{\Omega})} 
     + \sup_{\substack{t_1,t_2 \in [0,T] \\ t_1 \neq t_2}} \,\sup_{x \in \overline{\Omega}} \frac{|u(t_1,x) - u(t_2,x)|}{|t_1-t_2|^{\frac{1+\alpha}{2}}} .
\end{align*}
For open subset $\Omega \subset \mathbb{R}^n $ of class $C^{1,\alpha}$, we define the spaces
$C^{\frac{j+\alpha}{2}; j+\alpha}([0, T] \times \partial \Omega)$, $j \in \{0,1\}$, in the natural way by local parametrization of $\partial \Omega$. Similarly, we define the spaces $C^{j,\alpha}(\mathrm{M})$ and $C^{\frac{j+\alpha}{2}; j+\alpha}([0, T] \times \mathrm{M})$, $j \in \{0,1\}$, on a manifold $\mathrm{M}$ of class $C^{j,\alpha}$ embedded in $\mathbb{R}^n$ (see \cite[App.\,A]{DaLu23}). Finally, for $j \in \{0,1\}$, we set
\[
C_0^{\frac{j+\alpha}{2}; j+\alpha}([0, T] \times \overline{\Omega}):= 
\Big\{ u \in C^{\frac{j+\alpha}{2}; j+\alpha}([0, T] \times \overline{\Omega}) \colon u(0,x) = 0 \text{ for all } x \in \overline{\Omega} \Big\},
\]
and similarly for the spaces $C_0^{\frac{j+\alpha}{2}; j+\alpha}([0, T] \times \partial \Omega)$ and $C_0^{\frac{j+\alpha}{2}; j+\alpha}([0, T] \times \mathrm{M})$.

For functions in parabolic Schauder spaces, the partial derivative with respect to the space variable $x$ will be denoted by $\nabla$. At the same time, we will use the notation $\partial_t$ for the derivative with respect to the time variable $t$.

\subsection{Layer heat potentials}
In this section, we present some well-known facts about layer heat potentials. For proofs, we refer to Lady\v{z}enskaja, Solonnikov, and Ural'ceva
\cite{LaSoUr68}. 

\begin{definition}
    Let $S_n : \mathbb{R}^{1+n} \setminus
    \{(0,0)\}\to \mathbb{R}$ be the function defined by
    \[
    S_n(t,x):=
    \left\{
    \begin{array}{ll}
    \frac{1}{(4\pi t)^{\frac{n}{2}} }e^{-\frac{|x|^{2}}{4t}}&{\mathrm{if}}\ (t,x)\in ]0,+\infty[ \times{\mathbb{R}}^{n}\,, 
    \\
    0 &{\mathrm{if}}\ (t,x)\in (]-\infty,0]\times{\mathbb{R}}^{n})\setminus\{(0,0)\}.
    \end{array}
    \right.
    \]
    Then, $S_n$ is the fundamental solution of the heat operator $\partial_t-\Delta$ in $\mathbb{R}^{1+n} \setminus \{(0,0)\}$.
\end{definition}

We fix once and for all 
\[
\alpha \in ]0,1[, \, T>0 \text{ and } \Omega \text{ be an open bounded subset of }\mathbb{R}^n \text{ of class }C^{1,\alpha}.
\] 
Then, we introduce the single layer heat potential and the related boundary integral operator.

\begin{definition}\label{def single double potential}
    Let $\mu \in L^\infty\big([0,T] \times \partial\Omega\big)$. We define the single layer heat potential $v_{\Omega}[\mu]$ as the function of $[0,T]\times \mathbb{R}^n$ given by
    \begin{equation*} 
    v_{\Omega} [\mu](t,x) := \int_{0}^{t} \int_{\partial \Omega} S_n(t-\tau,x-y) \mu(\tau, y) \d \sigma_y \d \tau \quad \forall(t,x) \in [0, T] \times \mathbb{R}^n.
    \end{equation*}
    Moreover, we set 
    \begin{equation*}
        V_{\partial\Omega} [\mu] := \left( v_{\Omega} [\mu] \right)_{|[0,T]\times \partial\Omega},
    \end{equation*}
    and
    \begin{equation*}
    W^\ast_{\partial\Omega} [\mu](t,x) := \int_{0}^{t} \int_{\partial \Omega} \nabla S_n(t-\tau,x-y) \cdot \nu_\Omega(x) \mu(\tau, y) \d \sigma_y \d \tau \quad \forall(t,x) \in [0, T] \times \partial\Omega.
    \end{equation*}
\end{definition}

We stress that, in general, layer heat potentials satisfy properties similar to those of their standard elliptic counterparts. 
In the following, we collect some well-known properties of the single layer heat potential. For the proof of the following result, we refer to \cite{DaLu23,LaSoUr68,LaLu17,LaLu19}. In particular, a proof of Theorem \ref{thmsl}\,(iv) can be found in \cite{Mo24}.

\begin{theorem}\label{thmsl}
The following statements hold.
\begin{itemize}

\item[(i)] Let $\mu \in L^\infty([0,T] \times \partial\Omega)$. Then the function $v_{\Omega}[\mu]$ is continuous, $v_{\Omega}[\mu]$ solves the heat equation 
in $]0,T]\times (\mathbb{R}^n \setminus \partial\Omega)$ and 
$v_{\Omega}[\mu] \in C^\infty(]0,T[ \times (\mathbb{R}^n \setminus \partial\Omega))$. 

\item[(ii)] Let $v_\Omega^+[\mu]$ and $v_\Omega^-[\mu]$ denote 
the restrictions of $v_\Omega[\mu]$ to $[0,T] \times \overline{\Omega}$ and to $[0,T]\times \overline{\Omega^-}$, respectively. Then, the map from  $C_0^{\frac{\alpha}{2};  \alpha}([0,T] \times \partial\Omega)$ to  $C_{0}^{\frac{1+\alpha}{2}; 1+\alpha}([0,T] \times \overline{\Omega})$ that takes $\mu$ to $v_{\Omega}^+[\mu]$ is linear and continuous. If $R>0$ is such that $\overline{\Omega}$ is contained in the ball $B(0,R)$ of center $0$ and radius $R$, then the map from  $C_0^{\frac{\alpha}{2};  \alpha}([0,T] \times \partial\Omega)$ to  $C_{0}^{\frac{1+\alpha}{2}; 1+\alpha}([0,T] \times (\overline{B(0,R)}\setminus \Omega^-))$ that takes $\mu$ to $v_{\Omega}^-[\mu]_{|[0,T] \times (\overline{B(0,R)}\setminus \Omega^-)}$ is also linear and continuous.

\item[(iii)] Let $\mu \in C_0^{\frac{\alpha}{2};  \alpha}([0,T] \times \partial\Omega)$. Then the following jump relations hold: 
\[
\frac{\partial}{\partial \nu_\Omega}v_{\Omega}^\pm[\mu](t,x)  = \pm \frac{1}{2}\mu(t,x)
 +W_{\partial\Omega}^*[\mu](t,x), \qquad \forall (t,x) \in [0,T] \times \partial\Omega.
\]

\item[(iv)] The operator $V_{\partial \Omega}$ is an isomorphism from $C_0^{\frac{\alpha}{2}; \alpha}([0,T] \times \partial\Omega)$ to $C_0^{\frac{1+\alpha}{2}; 1+\alpha}([0,T] \times \partial\Omega)$.

\item[(v)] The operator $W^\ast_{\partial \Omega}$ is compact from $C_0^{\frac{\alpha}{2}; \alpha}([0,T] \times \partial\Omega)$ into itself.
\end{itemize}
\end{theorem}

\subsection{The class of diffeomorphisms and pullback layer potentials} \label{subsec: class of diffeo}
We introduce a class of diffeomorphisms that we use to model the shape of the perturbed domains. We recall some technical lemmas and smoothness (or even real analyticity) results related to the change of variables in integrals, extension operators in the neighborhood of the unperturbed boundary $\partial\Omega$, the pullback of the outer normal, and, finally, the smoothness of integral layer potential operators upon the parametrization of the support of integration.

So, let $\Omega$ be an open bounded connected subset of $\R^n$ of class $C^{1,\alpha}$, we define
\begin{equation*}
    \mathcal{A}_{\partial\Omega} := \left\{ \phi \in C^{1,\alpha}(\partial\Omega,\R^n) : \phi \text{ injective}, \, \d \phi(y) \text{ injective for all } y \in \partial\Omega \right\},    
\end{equation*}
and 
\begin{equation*}
    \mathcal{A}_{\overline{\Omega}} := \left\{ \phi \in C^{1,\alpha}(\overline{\Omega},\R^n) : \phi \text{ injective}, \, \d \phi(y) \text{ injective for all } y \in \overline{\Omega} \right\}.   
\end{equation*}

Then, it is well-known that the sets $\mathcal{A}_{\partial\Omega}$ and $\mathcal{A}_{\overline{\Omega}}$ are open in $C^{1,\alpha}(\partial\Omega,\R^n)$ and $C^{1,\alpha}(\overline{\Omega},\R^n)$,
respectively (see, e.g., Lanza de Cristoforis and Rossi \cite[Lem.\,2.2]{LaRo08}
and \cite[Lem.\,2.5]{LaRo04}). In particular, it is well established that if $\Omega$ has a connected exterior $\Omega^- = \R^n \setminus \overline{\Omega}$, then $\R^n \setminus \partial\Omega$ has two open connected components and thus the Jordan-Leray separation theorem ensures that $\R^n \setminus \phi(\partial\Omega)$ has exactly two open connected components for all $\phi \in \mathcal{A}_{\partial\Omega}$ (see, e.g., Deimling \cite[Thm.\,5.2]{Deimling85}). One of these open connected components is bounded, and we denote it by $\Omega[\phi]$, while the other one is unbounded, and we denote it by $\Omega[\phi]^-$. 

We start by recalling two technical lemmas that show that a diffeomorphism on $\partial\Omega$ can be extended in a suitably defined neighborhood of $\partial\Omega$ by means of a real-analytic extension operator (see \cite[Lems.\,2.2\,\&\,2.3]{DaLu23} and also \cite{LaRo08}).

\begin{lemma}\label{lemma Omega tub}
    Let $\Omega$ be an open bounded connected subset of $\R^n$ of class $C^{1,\alpha}$ such that $\Omega^-$ is also connected. There exists $\omega \in C^{1,\alpha}(\partial\Omega, \R^n)$, with $|\omega|=1$ and $\omega \cdot \nu_\Omega > \frac{1}{2}$ on $\partial\Omega$. Moreover, the following statements hold.
    \begin{enumerate}
        \item[(i)] There exists $\delta_{\Omega} \in (0,\infty)$ such that the sets
        \begin{equation}\label{eq: Omega omega delta}
        \begin{aligned}
            &\Omega_{\omega,\delta} := \{ x+s\omega(x) \colon x \in \partial\Omega, s \in (-\delta,\delta)\},
            \\
            &\Omega^+_{\omega,\delta} := \{ x+s\omega(x) \colon x \in \partial\Omega, s \in (-\delta,0)\},
            \\
            &\Omega^-_{\omega,\delta} := \{ x+s\omega(x) \colon x \in \partial\Omega, s \in (0,\delta)\},
        \end{aligned}
        \end{equation}
        are connected and of class $C^{1,\alpha}$. They have boundaries
        \begin{align*}
            &\partial\Omega_{\omega,\delta} := \{ x+s\omega(x) \colon x \in \partial\Omega, s \in \{-\delta,\delta\} \},
            \\
            &\partial\Omega^+_{\omega,\delta} := \{ x+s\omega(x) \colon x \in \partial\Omega, s \in \{-\delta,0\} \},
            \\
            &\partial\Omega^-_{\omega,\delta} := \{ x+s\omega(x) \colon x \in \partial\Omega, s \in \{0,\delta\} \},
        \end{align*}
        and we have $\Omega^+_{\omega,\delta} \subseteq \Omega$ and $\Omega^-_{\omega,\delta} \subseteq \Omega^-$ for all $\delta \in (0,\delta_{\Omega})$.
        
        \item[(ii)] Let $\delta \in (0,\delta_{\Omega})$. If $\Phi \in \mathcal{A}_{\overline{\Omega_{\omega,\delta}}}$, then $\phi := \Phi_{|\partial\Omega} \in \mathcal{A}_{\partial\Omega}$.

        \item[(iii)] If $\delta \in (0,\delta_{\Omega})$, then the set 
        \begin{equation*}
            \mathcal{A}^*_{\overline{\Omega_{\omega,\delta}}} := \left\{ \Phi \in \mathcal{A}_{\overline{\Omega_{\omega,\delta}}} \colon \Phi(\Omega^+_{\omega,\delta}) \subseteq \Omega[\Phi_{|\partial\Omega} ] \right\}
        \end{equation*}
        is open in $\mathcal{A}_{\overline{\Omega_{\omega,\delta}}}$ and $\Phi(\Omega^-_{\omega,\delta}) \subseteq \Omega[\Phi_{|\partial\Omega}]^-$ for all $\Phi \in \mathcal{A}^*_{\overline{\Omega_{\omega,\delta}}}$.

        \item[(iv)] If $\delta \in (0,\delta_{\Omega})$, then both $\Phi(\Omega^+_{\omega,\delta})$ and $\Phi(\Omega^-_{\omega,\delta})$ are open sets of class $C^{1,\alpha}$, and 
        \begin{equation*}
            \partial \Phi(\Omega^+_{\omega,\delta}) = \Phi(\partial \Omega^+_{\omega,\delta}), \quad \partial \Phi(\Omega^-_{\omega,\delta}) = \Phi(\partial \Omega^-_{\omega,\delta}).
        \end{equation*}
    \end{enumerate}
\end{lemma}

For the reader's convenience, we sketch a local representation of the domains $\Omega^\pm_{\omega,\delta}$ (in particular the boundaries $\partial \Omega^\pm_{\omega,\delta}$) built upon $\partial\Omega$ through the vector field $\omega$, see Figure \ref{fig: domains Omega omega,delta pm} below.

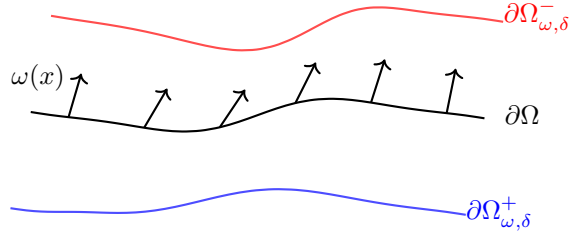
\begin{figure}[ht!]
\centering
\begin{tikzpicture}

\def\f{0.2*sin(1.2*\x r) + 0.05*sin(2.5*\x r)}

\def\del{1.3}

\draw[thick, domain=-3:3, smooth, variable=\x]
  plot ({\x}, {\f});

\foreach \x in {-2.5,-1.5,-0.5,0.5,1.5,2.5} {

    \pgfmathsetmacro{\y}{0.2*sin(1.2*\x r) + 0.05*sin(2.5*\x r)}

    \pgfmathsetmacro{\wx}{0.4 + 0.2*cos(\x r)}
    \pgfmathsetmacro{\wy}{1 + 0.3*sin(\x r)}

    \pgfmathsetmacro{\norm}{sqrt(\wx*\wx + \wy*\wy)}

    \pgfmathsetmacro{\ux}{\wx/\norm}
    \pgfmathsetmacro{\uy}{\wy/\norm}

    \draw[->, thick] (\x,\y) -- ++({0.6*\ux},{0.6*\uy});
}

\draw[red!70, thick, domain=-3:3, smooth, variable=\x]
plot ({
    \x + \del*(0.4 + 0.2*cos(\x r))
           /sqrt((0.4 + 0.2*cos(\x r))^2 + (1 + 0.3*sin(\x r))^2)
},{
    \f + \del*(1 + 0.3*sin(\x r))
           /sqrt((0.4 + 0.2*cos(\x r))^2 + (1 + 0.3*sin(\x r))^2)
});

\draw[blue!70, thick, domain=-3:3, smooth, variable=\x]
plot ({
    \x - \del*(0.4 + 0.2*cos(\x r))
           /sqrt((0.4 + 0.2*cos(\x r))^2 + (1 + 0.3*sin(\x r))^2)
},{
    \f - \del*(1 + 0.3*sin(\x r))
           /sqrt((0.4 + 0.2*cos(\x r))^2 + (1 + 0.3*sin(\x r))^2)
});

\node at (3.5,0) {\small$\partial\Omega$};
\node at (3.7,1.3) {\small\color{red}$\partial\Omega^-_{\omega,\delta}$};
\node at (3.2,-1.3) {\small\color{blue}$\partial\Omega^+_{\omega,\delta}$};
\node at (-2.9,0.5) {\small$\omega(x)$};

\end{tikzpicture}
\caption{A local visualization of the boundaries of the sets $\Omega^\pm_{\omega,\delta}$.}
\label{fig: domains Omega omega,delta pm}
\end{figure}

\begin{lemma}\label{lemma extension operator}
    Let $\Omega, \omega,\delta_{\Omega}$ be as in Lemma \ref{lemma Omega tub}. Let $\tilde{\phi} \in \mathcal{A}_{\partial\Omega}$. Then the following statements hold.
    \begin{enumerate}
        \item[(i)] There exist $\tilde{\delta} \in (0,\delta_{\Omega})$ and $\tilde{\Phi} \in \mathcal{A}^*_{\overline{\Omega_{\omega,\tilde{\delta}}}}$ such that $\tilde{\phi} = {\tilde{\Phi}}_{|\partial\Omega}$.

        \item[(ii)] Let $\tilde{\delta}, \tilde{\Phi}$ be as in statement (i). Then there exist an open neighborhood $\tilde{\mathcal{W}}$ of $\tilde{\phi}$ in $\mathcal{A}_{\partial\Omega}$, and a real analytic extension operator $\mathbf{E}$ from $C^{1,\alpha}(\partial\Omega,\R^n)$ to $C^{1,\alpha}(\overline{\Omega_{\omega,\tilde{\delta}}}, \R^n)$ which maps $\tilde{\mathcal{W}}$ to $\mathcal{A}^*_{\overline{\Omega_{\omega,\tilde{\delta}}}}$ and such that $\mathbf{E}[\tilde{\phi}] = \tilde{\Phi}$ and $\mathbf{E}[\phi]_{|\partial\Omega} = \phi$, for all $\phi \in \tilde{\mathcal{W}}$.
    \end{enumerate}
\end{lemma}

\begin{remark}\label{remark Phi_0}
    From a careful inspection of the proof of \cite[Props.\,2.7\,\&\,2.8]{LaRo08}, one can deduce that for $\tilde{\phi} = \phi_0=\mathrm{id}_{\partial\Omega^i}$ and under the assumption that $\overline{\Omega^i} \subset \Omega^o$, then it is always possible to choose the extension $\tilde{\Phi}$ provided by Lemma \ref{lemma extension operator} to be equal to $\mathrm{id}_{\overline{\Omega^i_{\omega,\tilde{\delta}}}}$. Moreover, possibly restricting the neighborhood $\tilde{\mathcal{W}}$ of $\mathrm{id}_{\partial\Omega^i}$ in $\mathcal{A}^*_{\overline{\Omega^i_{\omega,\tilde{\delta}}}}$, we can assume that $\mathbf{E}[\phi]( \overline{\Omega^i_{\omega,\tilde{\delta}}}) \subset \Omega^o$ for all $\phi \in \tilde{\mathcal{W}}$.
\end{remark}

Finally, we recall a real analytic result for change of variables in integrals and pullback of the outer normal vector field (see \cite[Prop.\,1]{La07}). 

\begin{lemma}\label{lemma change of variable}
    Let $\Omega$ be as in Lemma \ref{lemma Omega tub}. Then the following statements hold.
    \begin{enumerate}
        \item[(i)] For each $\phi \in \mathcal{A}_{\partial\Omega}$ there exists a unique positive $\tilde{\sigma}_n[\phi] \in C^{1,\alpha}(\partial\Omega)$ such that
        \begin{equation*}
            \int_{\phi(\partial\Omega)} f(s) \d \sigma_s = \int_{\partial\Omega} f \circ \phi(y) \,  \tilde{\sigma}_n[\phi](y) \d \sigma_y \quad \text{for all } f \in L^1(\phi(\partial\Omega)).
        \end{equation*}
        Moreover, the map that takes $\phi$ to $\tilde{\sigma}_n[\phi]$ is real analytic from $\mathcal{A}_{\partial\Omega}$ to $C^{0,\alpha}(\partial\Omega)$.

        \item[(ii)] The map from $\mathcal{A}_{\partial\Omega}$ to $C^{0,\alpha}(\partial\Omega, \R^n)$ that takes $\phi$ to $\nu_\phi \circ \phi$ is real analytic.
    \end{enumerate}
\end{lemma}

In order to shorten our notation, if $A$ is a subset of $\R^n$, and $h$ is a map from $A$ to $\R^n$, we will denote by $h^T$ the map from $[0,T] \times A$ to $[0,T] \times \R^n$ defined by
\begin{equation*}
    h^T(t,x) := (t, h(x)) \quad \forall (t,x) \in [0,T] \times A.
\end{equation*}
Let $\phi \in \mathcal{A}_{\partial\Omega}$ and $\mu \in C^0([0,T] \times \partial\Omega)$. To work with a space of densities that is not $\phi$-dependent, following \cite{DaLu23}, it makes sense to consider a density $\mu \circ (\phi^T)^{(-1)}$, and to define the single layer potential via the corresponding form given by Definition \ref{def single double potential}, namely,
\begin{align*}
    v_{\Omega}[\mu \circ (\phi^T)^{(-1)}](t,x) &:= \int_{0}^{t} \int_{\phi(\partial \Omega)} S_n(t-\tau,x-y) \mu \circ (\phi^T)^{(-1)} (\tau, y) \d \sigma_y \d \tau,
\end{align*}
for all $(t,x) \in [0, T] \times \mathbb{R}^n$. Moreover, we can define the boundary integral operator associated with the $\phi$-pullback of the single layer potential, that is 
\begin{equation*}
    V_{\phi,\partial\Omega}[\mu] := v_{\Omega}[\mu \circ (\phi^T)^{(-1)}] \circ \phi^T \quad \text{on } [0, T] \times \partial\Omega.
\end{equation*}
We also consider the operator defined by
\begin{equation*}
    W^\ast_{\phi,\partial\Omega}[\mu](t,x) :=  \int_{0}^{t} \int_{\phi(\partial \Omega)} \nabla S_n(t-\tau, \phi(x) -y) \cdot \nu_\phi(\phi(x)) \mu \circ (\phi^T)^{(-1)} (\tau, y) \d \sigma_y \d \tau,
\end{equation*}
for all $(t,x) \in [0, T] \times \partial\Omega$.

Finally, let $\Phi \in \mathcal{A}^*_{\overline{\Omega_{\omega,\tilde{\delta}}}}$ and let $\mu \in C^0([0,T] \times \partial\Omega)$. We define 
\begin{equation*}
     V^{\pm}_{\Phi,\partial\Omega}[\mu] := v^{\pm}_{\Omega}[\mu \circ (\Phi^T)^{(-1)}] \circ \Phi^T \quad \text{on } [0, T] \times \overline{\Omega^\pm_{\omega,\tilde{\delta}}}.
\end{equation*}

We now state two theorems on the dependence of layer heat potentials upon perturbation of the support and of the density, which will be used in the application in Section \ref{sec: application} of our main result to boundary value problems. For the first theorem, we refer to \cite[Thm.\,5.3]{DaLu23}, while for the second theorem, we refer to \cite[Thm.\,5.4]{DaLu23}.

\begin{theorem}\label{thm dependence V Phi W Phi}
    Let $\Omega, \omega,\delta_{\Omega}$ be as in Lemma \ref{lemma Omega tub}. Then, the map that takes $\Phi \in \mathcal{A}^*_{\overline{\Omega_{\omega,\tilde{\delta}}}}$ to 
    \begin{equation*}
        V_{\Phi,\partial\Omega}^{\pm} \in \mathcal{L} \left( C^{\frac{\alpha}{2}, \alpha}_0([0,T] \times \partial\Omega), C^{\frac{1+\alpha}{2}, 1+\alpha}_0([0,T] \times \overline{\Omega^{\pm}_{\omega,\tilde{\delta}}}) \right)
    \end{equation*}
    is of class $C^\infty$.
\end{theorem}

\begin{theorem}\label{thm dependence V phi W phi}
    Let $\Omega, \omega,\delta_{\Omega}$ be as in Lemma \ref{lemma Omega tub}. Then the following statements hold. 
    \begin{itemize}
        \item[(i)] The map that takes $\phi \in \mathcal{A}_{\partial \Omega}$ to 
        \begin{equation*}
            V_{\phi,\partial\Omega} \in \mathcal{L} \left( C^{\frac{\alpha}{2}, \alpha}_0([0,T] \times \partial\Omega), C^{\frac{1+\alpha}{2}, 1+\alpha}_0([0,T] \times \partial\Omega) \right)
        \end{equation*}
        is of class $C^\infty$.  

        \item[(ii)] The map that takes $\phi \in \mathcal{A}_{\partial \Omega}$ to 
        \begin{equation*}
            W^\ast_{\phi,\partial\Omega} \in \mathcal{L} \left( C^{\frac{\alpha}{2}, \alpha}_0([0,T] \times \partial\Omega), C^{\frac{\alpha}{2}, \alpha}_0([0,T] \times \partial\Omega) \right)
        \end{equation*}
        is of class $C^\infty$.   
    \end{itemize}
\end{theorem}

\section{Shape analysis of the energy: abstract result}\label{sec: abstract result}

Let $\Omega^o$ and $\Omega^i$ be as in \eqref{introsetconditions}. To shorten our notation, we just set 
\[
\Omega^{i,\pm}_{\omega,\tilde{\delta}} := \left( \Omega^{i}\right)^\pm_{\omega,\tilde{\delta}},
\]
where the right-hand side is given by \eqref{eq: Omega omega delta} for $\Omega = \Omega^i$ (see Figure \ref{fig: domains Omega^o,i}). 

\begin{figure}[ht!]
    \centering
    \begin{tikzpicture}

    \fill[gray!30]
      plot[smooth cycle, tension=0.8]
      coordinates {(0,0) (4,0.5) (5,2.5) (3.5,4) (1,3.5) (-0.5,2)};
    \draw[thick, black]
      plot[smooth cycle, tension=0.8]
      coordinates {(0,0) (4,0.5) (5,2.5) (3.5,4) (1,3.5) (-0.5,2)};

    \def\inner{
      (1,1) (3,1.2) (3.5,2.2) (2.5,3) (1.2,2.5)
    }
    \def\innerdelta{
      (1.2,1.2) (2.8,1.35) (3.3,2.15) (2.45,2.85) (1.3,2.35)
    }
    \def\outerdelta{
      (0.8,0.8) (3.25,1.1) (3.65,2.3) (2.7,3.15) (1.1,2.7)
    }

    \fill[gray!60, even odd rule]
      plot[smooth cycle, tension=0.8]
      coordinates \inner 
      plot[smooth cycle, tension=0.8]
      coordinates \outerdelta; 

    \fill[white]
      plot[smooth cycle, tension=0.8]
      coordinates \inner;

    \fill[gray!95, even odd rule]
      plot[smooth cycle, tension=0.8]
      coordinates \inner
      plot[smooth cycle, tension=0.8]
      coordinates \innerdelta;

    \draw[thick, black]
      plot[smooth cycle, tension=0.8]
      coordinates \inner;

    \node at (4.2,3) {\tiny$\Omega^o$};
    \node at (1.5,1.6) {\tiny$\Omega^i$};
    \node at (3.6,1) {\tiny$\Omega^{i,-}_{\omega,\delta}$};
    \node at (2.3,2.6) {\tiny$\Omega^{i,+}_{\omega,\delta}$};

    \end{tikzpicture}
    \caption{The neighbourhoods $\Omega^{i,\pm}_{\omega,\tilde{\delta}}$ of the domain $\Omega^i$.}
    \label{fig: domains Omega^o,i}
\end{figure}
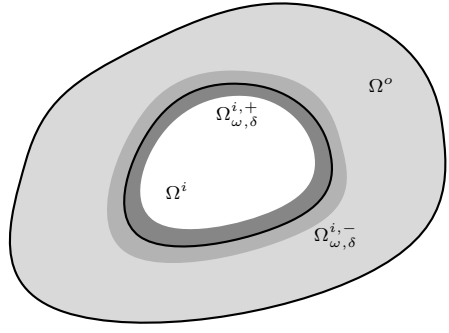 

In order to prove the abstract result on the dependence of the energy upon the perturbation parameter we need a technical lemma, which provides the existence of a regular diffeomorphism that maps $\overline{\Omega^o}\setminus\Omega^i$ onto $\overline{\Omega^o} \setminus \Omega^i[\phi]$ which depends real analytically on $\phi$: it is built upon a suitable extension to the whole domain $\overline{\Omega^o}$ of the extension $\mathbf{E}[\phi]$ provided by Lemma \ref{lemma extension operator}.

\begin{lemma}\label{lemma diffeo}
    Let $\omega,\delta_{\Omega^i}$ be as in Lemma \ref{lemma Omega tub} for $\Omega = \Omega^i$ and for $\tilde{\phi}_0 = \phi_0 =\mathrm{id}_{\partial\Omega^i}$. Then, there exist an open neighborhood $\mathcal{W}^\ast_0$ of $\phi_0$ in $\mathcal{A}^{\Omega^o}_{\partial\Omega^i}$, positive numbers
    \[
    0<\rho_1<\rho_2<\tilde\delta<\delta_{\Omega^i},
    \]
    and a real analytic extension operator $\Psi$ from $\mathcal{W}^\ast_0$ to $C^{1,\alpha}(\overline{\Omega^o}, \R^n)$ which takes $\phi$ to $\Psi[\phi]$, such that the following properties hold:
    \begin{itemize}
        \item[(i)] $\Psi[\phi]$ is a $C^{1,\alpha}$-diffeomorphism of $\overline{\Omega^o}$ onto itself, for all $\phi \in \mathcal{W}^\ast_0$.
        \item[(ii)] $\Psi[\phi] = \mathbf{E}[\phi]$ on $\overline{\Omega^i_{\omega,\rho_1}}$ and $\Psi[\phi] = \mathrm{id}_{\overline{\Omega^o}}$ on $\overline{\Omega^o} \setminus \Omega^i_{\omega,\rho_2}$, for all $\phi \in \mathcal{W}^\ast_0$.
        \item[(iii)] $\Psi[\phi](\Omega^i) = \Omega^i[\phi]$ and $\Psi[\phi](\Omega^o\setminus\overline{\Omega^i}) = \Omega^o\setminus\overline{\Omega^i[\phi]}$, for all $\phi \in \mathcal{W}^\ast_0$.
    \end{itemize}
\end{lemma}

\begin{proof}
    By Lemma \ref{lemma extension operator} there exist $\tilde{\delta}\in(0,\delta_{\Omega^i})$, an open neighborhood $\tilde{\mathcal{W}}_0$ of $\phi_0$ in $\mathcal{A}^{\Omega^o}_{\partial\Omega^i}$, and a real analytic extension operator $\mathbf{E}$ from $\tilde{\mathcal{W}}_0$ to $\mathcal{A}^\ast_{\overline{\Omega^i_{\omega,\tilde{\delta}}}}$ such that $\mathbf{E}[\phi]_{|\partial\Omega^i}=\phi$ for every $\phi\in\tilde{\mathcal{W}}_0$ and $\mathbf{E}[\phi_0]=\mathrm{id}_{\overline{\Omega^i_{\omega,\tilde{\delta}}}}$ on $\overline{\Omega^i_{\omega,\tilde{\delta}}}$, see also Remark \ref{remark Phi_0}. We also assume, after possibly shrinking $\tilde{\mathcal{W}}_0$, that
    \[
        \mathbf{E}[\phi](\overline{\Omega^i_{\omega,\tilde{\delta}}})\subset \Omega^o
        \qquad \forall \phi\in\tilde{\mathcal{W}}_0.
    \]
    Fix $0<\rho_1<\rho_2<\tilde\delta$ and $\chi\in C^{1,\alpha}(\overline{\Omega^o})$ such that
    \[
    \chi=1 \text{ on } \overline{\Omega^i_{\omega,\rho_1}},
    \qquad
    \chi=0 \text{ on } \overline{\Omega^o}\setminus \Omega^i_{\omega,\rho_2},
    \qquad
    0\leq\chi\leq 1 \text{ on } \overline{\Omega^o}.
    \]
    For $\phi\in\tilde{\mathcal{W}}_0$, we define
    \begin{equation*}
        \Psi[\phi](x) :=
        \begin{cases}
        x & \text{for } x\in \overline{\Omega^o}\setminus\Omega^i_{\omega,\tilde\delta},\\
        x+\chi(x)\bigl(\mathbf{E}[\phi](x)-x\bigr) & \text{for } x\in \overline{\Omega^i_{\omega,\tilde\delta}}.
        \end{cases}
    \end{equation*}
    Then
    \begin{equation}\label{eq Psi phi 2}
        \Psi[\phi](x)=
        \begin{cases}
        x & \text{for } x\in \overline{\Omega^o}\setminus\Omega^i_{\omega,\rho_2},\\
        x+\chi(x)\bigl(\mathbf{E}[\phi](x)-x\bigr) & \text{for } x\in \overline{\Omega^i_{\omega,\rho_2}}\setminus \Omega^i_{\omega,\rho_1}\\
        \mathbf{E}[\phi](x), & \text{for } x\in \overline{\Omega^i_{\omega,\rho_1}}.
        \end{cases}
    \end{equation}
    We claim that there exists an open neighborhood $\mathcal{W}^\ast_0$ of $\tilde{\phi}_0$ in $\mathcal{A}^{\Omega^o}_{\partial\Omega^i}$ such that the previous map defines a diffeomorphism of $\overline{\Omega^o}$ into itself and $\Psi$ is the real analytic extension operator. For each fixed $\phi \in \tilde{\mathcal{W}}_0$ and for each $x\in \overline{\Omega^o}$ the Jacobian matrix $D\Psi[\phi](x) \in \R^{n\times n}$ is given by
    \begin{equation*}
        D \Psi[\phi](x)=
        \begin{cases}
         I_n & x \in \overline{\Omega^o} \setminus \Omega^i_{\omega,\rho_2},
         \\
         I_n +\chi(x)\bigl(D \mathbf{E}[\phi](x) - I_n \bigr) + \nabla \chi(x) \otimes \bigl(\mathbf{E}[\phi](x) - x \bigr)  & x \in \overline{\Omega^i_{\omega,\rho_2}} \setminus \Omega^i_{\omega,\rho_1},
         \\
         D \mathbf{E}[\phi](x) & x \in \overline{\Omega^i_{\omega,\rho_1}},
        \end{cases}
    \end{equation*}
    where $I_n$ denotes the identity matrix in $\R^{n\times n}$ and $\otimes$ denotes the tensor product of $\R^n$.
    
    In particular, $\Psi[\phi]\in C^{1,\alpha}(\overline{\Omega^o},\R^n)$. In fact, the continuity of $\Psi[\phi]$ and of $D\Psi[\phi]$ across the interfaces follows from the choice of $\chi$ and from the identities $\chi=0$ on $\partial\Omega^i_{\omega,\rho_2}$ and $\chi=1$ on $\partial\Omega^i_{\omega,\rho_1}$.

    Moreover, by the real analyticity of the operator $\mathbf{E}$ from $\tilde{\mathcal{W}}_0$ to $\mathcal{A}^\ast_{\overline{\Omega^i_{\omega,\tilde{\delta}}}}$, and since $\mathbf{E}[\phi_0] = \mathrm{id}_{\overline{\Omega^i_{\omega,\tilde{\delta}}}}$, then the continuous dependence with respect to the uniform convergence of both $\mathbf{E}[\phi]$ and $D\mathbf{E}[\phi]$ yields that
    \begin{equation*}
        \| D\Psi[\phi] - I_n \|_{C^{0}(\overline{\Omega^o}, \R^{n\times n})} < \frac{1}{2} \quad \forall \phi \in \mathcal{W}^\ast_0,
    \end{equation*}
    for an open neighborhood $\mathcal{W}^\ast_0$ of ${\phi}_0$ in $\mathcal{A}^{\Omega^o}_{\partial\Omega^i}$. Then, $D\Psi[\phi](x)$ is invertible for every $x\in\overline{\Omega^o}$, since
    \[
        |D\Psi[\phi](x)\xi|\geq |\xi|-|(D\Psi[\phi](x)-I_n)\xi|\geq \frac12|\xi|.
    \]
    Thus $\Psi[\phi]$ is a local $C^{1,\alpha}$-diffeomorphism.

    We now prove global injectivity. In order to do that, we actually prove that $\Psi[\phi]$ is bi-Lipschitz. Since $\overline{\Omega^o}$ is compact, connected and of class $C^{1,\alpha}$, there exists a constant $C_{\Omega^o} \geq 1$ such that 
    \[
    d_{\Omega^o}(x,z) \leq C_{\Omega^o}|x-z| \quad \text{for all } x,z\in\overline{\Omega^o},
    \] 
    where $d_{\Omega^o}(x,z)$ is the geodesic distance in $\overline{\Omega^o}$ between $x,z \in \overline{\Omega^o}$, defined by
    \[
    d_{\Omega^o}(x,z) := \inf \{ l(\gamma) \colon \gamma \in C^1([0,1]; \overline{\Omega^o}), \gamma(0) = x, \gamma(1) = z \}, \quad l(\gamma):= \int_{0}^{1} |\gamma'(s)| \d s.
    \]
    By further shrinking $\mathcal{W}^\ast_0$ if necessary, we can assume that 
    \begin{equation}\label{ineq D Psi[phi] - I  < c}
        \| D\Psi[\phi] - I_n \|_{C^{0}(\overline{\Omega^o}, \R^{n\times n})} < c \quad \forall \phi \in \mathcal{W}^\ast_0,
    \end{equation}
    where $c \in (0,1)$ is a suitable constant such that $ c < \min \left\{\frac{1}{2}, \frac{1}{C_{\Omega^o}} \right\} $. By triangle inequality,
    \[
    \|D \Psi[\phi] \|_{C^{0}(\overline{\Omega^o}, \R^{n\times n})} \leq \| D\Psi[\phi] - I_n \|_{C^{0}(\overline{\Omega^o}, \R^{n\times n})} + \|I_n \|_{C^{0}(\overline{\Omega^o}, \R^{n\times n})} \leq 1+c.
    \]
    
    A direct application of the Fundamental Theorem of Calculus along a geodesic $\gamma$ connecting $x,z \in \overline{\Omega^o}$ provides
    \begin{equation*}
    \begin{aligned}
      &\Psi[\phi](x)-\Psi[\phi](z) = \int_{0}^{1} D \Psi[\phi](\gamma(s)) \gamma'(s) \d s ,
      \\
      &\left( \Psi[\phi] - \mathrm{id}_{\overline{\Omega^o}} \right)(x) - \left( \Psi[\phi] - \mathrm{id}_{\overline{\Omega^o}} \right)(z)  = \int_{0}^{1} \left( D \Psi[\phi]- I_n \right)(\gamma(s)) \gamma'(s) \d s .
    \end{aligned}
    \end{equation*}
    Hence, for the upper bound, we conclude that 
    \begin{equation*}
        |\Psi[\phi](x)-\Psi[\phi](z)|\leq (1+c) C_{\Omega^o} |x-z| \quad \forall x,z\in\overline{\Omega^o}.
    \end{equation*}
    On the other hand
    \begin{equation*}
        \left| \left( \Psi[\phi] - \mathrm{id}_{\overline{\Omega^o}} \right)(x) - \left( \Psi[\phi] - \mathrm{id}_{\overline{\Omega^o}} \right)(z) \right| \leq c \, C_{\Omega^o} |x-z| \quad \forall x,z\in\overline{\Omega^o},
    \end{equation*}
    and, consequently,
    \begin{equation*}
    \begin{aligned}
        |\Psi[\phi](x)-\Psi[\phi](z)| &= \left| x-z + \left( \Psi[\phi] - \mathrm{id}_{\overline{\Omega^o}} \right)(x) - \left( \Psi[\phi] - \mathrm{id}_{\overline{\Omega^o}} \right)(z) \right|
        \\
        &\geq
        |x-z| - \left| \left( \Psi[\phi] - \mathrm{id}_{\overline{\Omega^o}} \right)(x) - \left( \Psi[\phi] - \mathrm{id}_{\overline{\Omega^o}} \right)(z) \right|
        \\
        &\geq 
        \left( 1 - c  \, C_{\Omega^o} \right) |x-z| \qquad \forall x,z\in\overline{\Omega^o},
    \end{aligned}
    \end{equation*}
    where, by \eqref{ineq D Psi[phi] - I  < c}, the coefficient in the last line is strictly positive.
    
    Hence, $\Psi[\phi]$ is bi-Lipschitz and injective. It remains to verify surjectivity onto $\overline{\Omega^o}$. Since $\Psi[\phi]$ is a local diffeomorphism in $\Omega^o$, the set $\Psi[\phi](\Omega^o)$ is open in $\Omega^o$. Moreover, $\Psi[\phi](\overline{\Omega^o})$ is compact and hence closed in $\R^n$, while injectivity together with $\Psi[\phi](\partial\Omega^o)=\partial\Omega^o$ gives $\Psi[\phi](\Omega^o)=\Psi[\phi](\overline{\Omega^o})\cap\Omega^o$. Thus $\Psi[\phi](\Omega^o)$ is also closed in the connected set $\Omega^o$. Since it is nonempty, $\Psi[\phi](\Omega^o)=\Omega^o$, and consequently $\Psi[\phi](\overline{\Omega^o})=\overline{\Omega^o}$. Therefore $\Psi[\phi]$ is a global $C^{1,\alpha}$-diffeomorphism. In particular, standard regularity of inverses of $C^{1,\alpha}$-diffeomorphisms yields that the global inverse is also of class $C^{1,\alpha}$.
    
    Moreover, by definition, $\Psi[\phi]( \partial \Omega^o) = \partial\Omega^o$ and $\Psi[\phi]( \partial \Omega^i) = \phi(\partial\Omega^i)$. Furthermore, since $\Psi[\phi]$ is a homeomorphism, it maps the open connected subset $\Omega^i$ onto an open connected bounded subset of $\mathbb R^n$. Since the bounded connected component of $\mathbb R^n\setminus \phi(\partial\Omega^i)$ is unique by the Jordan-Leray separation Theorem and by definition it is $\Omega^i[\phi]$, we conclude that $\Psi[\phi](\Omega^i) = \Omega^i[\phi]$. Since $\Psi[\phi]$ maps $\overline{\Omega^o}$ onto itself, it follows that $
    \Psi[\phi](\Omega^o\setminus\overline{\Omega^i})=\Omega^o\setminus\overline{\Omega^i[\phi]}$.
    
    Finally, the map from $\mathcal{W}^\ast_0$ to $C^{1,\alpha}(\overline{\Omega^o}, \mathbb{R}^n)$ that takes $\phi$ to $\Psi[\phi]$ is real analytic since it is obtained from the real analytic map $\mathbf{E}$ by composition with the bounded linear operator from $C^{1,\alpha}(\overline{\Omega^o}, \mathbb{R}^n)$ into itself that takes a map $\Phi \in C^{1,\alpha}(\overline{\Omega^o}, \mathbb{R}^n)$ to the function of the variable $x \in \overline{\Omega^o}$ given by $x + \chi(x)(\Phi(x) - x)$, hence proving the lemma.
\end{proof}

We finally state the elementary fixed-domain integration result that will be used in the proof of our main theorem below. 

\begin{lemma}\label{lemma integral regularity}
    Let $\Omega$ be an open bounded subset of $\R^n$ of class $C^{1,\alpha}$. The map
    \[
    \mathcal{I} \colon C_0^{\frac{\alpha}{2}; \alpha}([0, T] \times \overline{\Omega}) \times C^{0, \alpha}(\overline{\Omega}) \longrightarrow C^{0}([0,T])
    \]
    that takes a pair $(v,\rho)$ to the function
    \[
    \mathcal{I}[v,\rho](t)=\int_\Omega v(t,x)^2\rho(x) \d x
    \]
    is of class $C^\infty$.
\end{lemma}

We are now ready to state and prove the main result.

\begin{proof}[Proof of Theorem \ref{thm energy dependence abstract}]
For every fixed $\phi \in \widehat{\mathcal{W}}_0$, Dominated Convergence Theorem and assumption (i) imply $e_\phi \in C^0([0,T])$. The standard energy derivation, justified by the approximation argument in \cite[Lem.\,5\,\&\,Prop.\,2]{Lu20} and by the Divergence Theorem, gives for $t\in(0,T)$
    \begin{equation*}
    \begin{aligned}
        \frac{\d{}}{\d t}e_\phi(t) &= \int_{\Omega^o \setminus \overline{\Omega^i[\phi]}} \partial_t u_\phi(t,y) \, u_\phi(t,y) \d y = \int_{\Omega^o \setminus \overline{\Omega^i[\phi]}} \Delta u_\phi(t,y) \, u_\phi(t,y) \d y 
        \\
        &= - \int_{\Omega^o \setminus \overline{\Omega^i[\phi]}} |\nabla u_\phi(t,y)|^2 \d y + \int_{\partial\Omega^o} 
        \nabla u_\phi(t,y)\cdot \nu_{\Omega^o}(y) \, u_{\phi}(t,y) \d \sigma_y 
        \\
        & \quad - \int_{\phi(\partial\Omega^i)} 
        \nabla u_\phi(t,\tilde{y}) \cdot \nu_{\Omega^i[\phi]}(\tilde{y})  \, u_{\phi}(t, \tilde{y}) \d \sigma_{\tilde{y}},
    \end{aligned}    
    \end{equation*}
    where we used the equation \eqref{princeqpertu family} solved by $u_\phi$ on $\Omega^o \setminus \overline{\Omega^i[\phi]}$ for every $\phi \in \widehat{\mathcal{W}}_0$. Lemma \ref{lemma change of variable} yields \eqref{eq: derivative of e general} for $t\in(0,T)$. Furthermore, \eqref{eq: derivative of e general} and the regularity of $u_\phi$ imply that $\frac{\d{}}{\d t}e_\phi$ has a unique continuous extension to $[0,T]$, hence $e_\phi \in C^1([0,T])$.

    We now prove the smoothness of the map $e_\phi$ upon the parameter $\phi$. To better explain our argument, we first establish smoothness considering $e_\phi$ a $C^0([0,T])$-valued map. Then, we prove smoothness upon the parameter $\phi$ of $\frac{\d{}}{\d t}e_\phi$ as a $C^0([0,T])$-valued map.

    We invoke Lemma \ref{lemma diffeo} and retain the notation introduced there. For $\phi \in \widehat{\mathcal{W}}_0$, we define
    \[
    \widehat\Omega:=\Omega^o\setminus\overline{\Omega^i},
    \qquad
    \widehat\Omega[\phi]:=\Omega^o\setminus\overline{\Omega^i[\phi]}.
    \]
    For every $\phi\in \widehat{\mathcal{W}}_0$ let us define $U_\phi$ the pulled back function on the fixed domain $\widehat\Omega$  given by \eqref{eq: U_phi := u_phi Psi phi}. By Lemma \ref{lemma diffeo}, the map $\Psi[\phi]$ is a diffeomorphism from $\widehat\Omega$ into $\widehat\Omega[\phi]$. Therefore, the standard Change of Variables Theorem yields for every $t \in [0,T]$
    \begin{equation}\label{eq: energy pull back}
        e_\phi(t)= \frac{1}{2} \int_{\widehat\Omega} U_\phi(t,x)^2\,J[\phi](x) \d x,
     \qquad
     \text{ with } J[\phi](x):=\det D\Psi[\phi](x) \text{ for } x \in \widehat\Omega.
    \end{equation}
    Notice that $J[\phi]>0$ in $\widehat\Omega$ for every $\phi \in \widehat{\mathcal{W}}_0$. Moreover, by the real analyticity of the map from $\mathcal{W}^*_0$ to $C^{1,\alpha}(\overline{\Omega^o},\R^n)$ which takes $\phi$ to $\Psi[\phi]$ and since the determinant is a polynomial function with respect to the entries of the Jacobian matrix $D \Psi[\phi]$, we conclude that
    \begin{equation}\label{eq J real anal}
            \text{the map from $\mathcal{W}^*_0$ to $C^{0,\alpha}(\overline{\Omega^o})$ which takes $\phi$ to $J[\phi]$ is real analytic.}
    \end{equation}
    Combining \eqref{eq: energy pull back}, assumption (ii), \eqref{eq J real anal} and Lemma \ref{lemma integral regularity}, we conclude that the map from $\widehat{\mathcal{W}}_0$ to $C^0([0,T])$ which takes $\phi$ to $e_\phi$ is of class $C^\infty$.

    In this final step, we describe how to obtain the statement in $C^1([0,T])$. Set
    \begin{equation*}
    \begin{split}
        R_\phi(t)
        :=
        - \int_{\Omega^o \setminus \overline{\Omega^i[\phi]}} |\nabla u_\phi(t,y)|^2 \d y 
        + \int_{\partial\Omega^o} \nabla u_\phi(t,y)\cdot \nu_{\Omega^o}(y) \, u_{\phi}(t,y) \d \sigma_y &
        \\
         \quad
        - \int_{\partial\Omega^i} \nabla u_\phi(t,\phi(y))\cdot \nu_{\Omega^i[\phi]}(\phi(y)) \, u_\phi(t,\phi(y)) \, \tilde\sigma_n[\phi](y) \d\sigma_y & \quad \text{for all } t\in [0,T].
    \end{split}        
    \end{equation*}
    We now prove that the map from $\widehat{\mathcal W}_0$ to $C^0([0,T])$ which takes $\phi$ to $R_\phi$ is of class $C^\infty$. 

    By Change of Variables Theorem, by Lemma \ref{lemma diffeo} and by \eqref{eq: U_phi := u_phi Psi phi}, we obtain 
    \[
        \int_{\Omega^o\setminus\overline{\Omega^i[\phi]}}|\nabla_y u_\phi(t,y)|^2 \d y
        =
        \int_{\widehat\Omega}\left|D\Psi[\phi](x)^{-\top}\nabla_x U_\phi(t,x)\right|^2  J[\phi](x) \d x,
    \]
    where we have used the chain rule
    \[
        \nabla_y u_\phi(t,\Psi[\phi](x))
        = D\Psi[\phi](x)^{-\top}\nabla_x U_\phi(t,x) \quad \text{for all } (t,x) \in [0,T] \times \widehat\Omega.
    \]
    By assumption (ii) and by the linear continuity of the gradient operator, the map from $\widehat{\mathcal W}_0$ to $C_0^{\frac\alpha2;\alpha}([0,T]\times\overline{\widehat\Omega})$ which takes $\phi$ to $\nabla_x U_\phi$ is of class $C^\infty$. Combining that with Lemma \ref{lemma diffeo}, one obtains that the map from $\widehat{\mathcal{W}}_0$ to $C_0^{\frac{\alpha}{2}; \alpha}([0,T] \times \overline{\widehat\Omega})$ which takes $\phi$ to the function
    \[
    [0,T] \times \overline{\widehat\Omega} \ni (t,x) \longmapsto  \left|D\Psi[\phi](x)^{-\top}\nabla_x U_\phi(t,x)\right|^2 \in \R
    \]
    is of class $C^\infty$. Hence, by \eqref{eq J real anal} and by Lemma \ref{lemma integral regularity}, one obtains that the map from $\widehat{\mathcal{W}}_0$ to $C^0([0,T])$ which takes $\phi$ to the function of the variable $t$ given by
    \[
    \int_{\widehat{\Omega}} \left|D\Psi[\phi](x)^{-\top}\nabla_x{U}_\phi(t,x)\right|^2J[\phi](x) \d x 
    \]
    is of class $C^\infty$.

    We now treat the two boundary terms. Since $\Psi[\phi] = \mathrm{id}_{\overline{\Omega^o}}$ on $\overline{\Omega^o} \setminus \Omega^i_{\omega,\rho_2}$, one has
    \[
        u_{\phi}(t,y) = U_{\phi}(t,y), \quad \nabla_y u_\phi(t,y) = \nabla_x U_\phi(t,y)
        \quad \forall (t,y) \in [0,T] \times \overline{\Omega^o} \setminus \Omega^i_{\omega,\rho_2}.
    \]
    By assumption (ii), by the linear continuity of the gradient operator, and since the trace operator from $C_0^{\frac\alpha2;\alpha}([0,T]\times\overline{\widehat\Omega})$ to $C_0^{\frac\alpha2;\alpha}([0,T]\times\partial\Omega^o)$ is linear and continuous, we conclude that the map from $\widehat{\mathcal{W}}_0$  to $C_0^{\frac\alpha2;\alpha}([0,T]\times\partial\Omega^o)$ which takes $\phi$ to the function of variables $(t,y)$ given by $\nabla u_\phi(t,y)\cdot\nu_{\Omega^o}(y) \, u_{\phi}(t,y)$ is of class $C^\infty$. Hence, the map from $\widehat{\mathcal{W}}_0$ to $C^0([0,T])$ which takes $\phi$ to the function of the variable $t$ given by
    \[
    \int_{\partial\Omega^o} \nabla u_\phi(t,y)\cdot\nu_{\Omega^o}(y) \, u_{\phi}(t,y) \d\sigma_y
    \]
    is of class $C^\infty$.

    Finally, using the fact that $\Psi[\phi]=\mathbf{E}[\phi]$ on $\overline{\Omega^i_{\omega,\rho_1}}$ we have 
    \[
        u_{\phi}(t,\mathbf{E}[\phi](y)) = U_{\phi}(t,y),\quad \nabla_y u_\phi(t,\mathbf{E}[\phi](y))
        =D\Psi[\phi](y)^{-\top}\nabla_x U_\phi(t,y)
        \quad \forall (t,y) \in [0,T] \times \overline{\Omega^i_{\omega,\rho_1}}.
    \]
    Then, by $\mathbf{E}[\phi]_{|\partial\Omega^i}=\phi$, by Lemmas \ref{lemma change of variable} \& \ref{lemma diffeo}, by assumption (ii), by the linear continuity of the gradient operator, and since the trace operator from $C_0^{\frac\alpha2;\alpha}([0,T]\times\overline{\widehat\Omega})$ to $C_0^{\frac\alpha2;\alpha}([0,T]\times\partial\Omega^i)$ is linear and continuous, we conclude that the map from $\widehat{\mathcal{W}}_0$ to $C^0([0,T])$ which takes $\phi$ to the function of the variable $t$ given by
    \[
    \int_{\partial\Omega^i}\nabla u_\phi(t,\phi(y))\cdot\nu_{\Omega^i[\phi]}(\phi(y)) \, u_\phi(t,\phi(y)) \, \tilde\sigma_n[\phi](y) \d\sigma_y
    \]
    is of class $C^\infty$.

    We have shown that $\phi\mapsto R_\phi$ is $C^\infty$ from $\widehat{\mathcal{W}}_0$ to $C^0([0,T])$. Since $e_\phi'=R_\phi$ on $[0,T]$ and $e_\phi(0)=0$, we have
    \[
        e_\phi(t)=\int_0^t R_\phi(s) \d s \qquad \forall t\in[0,T].
    \]
    Therefore, the proof is complete.
\end{proof}

\section{Application to boundary value problems}\label{sec: application}
This section is devoted to showing how Theorem \ref{thm energy dependence abstract} can be applied to the solutions of different parabolic boundary value problems. To do that, we first introduce the two problems and we recall some results by \cite{DaLuMoMu24.2} and by \cite{DaLuMu25} on the existence and shape analysis of the solutions of these perturbed problems. We start by recalling a representation result from \cite[Lem.\,1]{DaLuMoMu24.2} (cf. Theorem \ref{thmsl}\,(iv)).

\begin{lemma}\label{lemma rappr}
    Let $\Omega$ be a bounded open connected subset of $\mathbb{R}^n$ of class $C^{1,\alpha}$, such that $\Omega^-$ is connected and $\overline{\Omega}\subseteq\Omega^o$. Then, the map from $C_0^{\frac{\alpha}{2}; \alpha}([0,T] \times \partial\Omega^o) \times C_0^{\frac{\alpha}{2}; \alpha}([0,T] \times \partial \Omega)$ to the space 
    \begin{equation*}
    \mathcal{X}_{\mathrm{heat}} := \left\{u\in C_{0}^{\frac{1+\alpha}{2}; 1+\alpha}([0,T] \times (\overline{\Omega^o} \setminus \Omega)) \colon \partial_t u  - \Delta u =0 \text{ in } ]0,T] \times \Omega^o \setminus \overline{\Omega} \right\}.
    \end{equation*}
    that takes a pair $(\xi,\gamma)$ to the function $u_{\Omega^o,\Omega}[\xi,\gamma]$ defined by
    \begin{equation}\label{eq: representation single layer}
    u_{\Omega^o,\Omega}[\xi,\gamma] := \left(v^+_{\Omega^o} [\xi] + v^-_{\Omega}[\gamma] \right)_{| [0,T] \times (\overline{\Omega^o} \setminus \Omega)},
    \end{equation}
    is bijective.
\end{lemma}

\begin{remark}
    Lemma \ref{lemma rappr} provides a representation for caloric functions on annular domains by means of sums of single layer heat potentials. It is well known that one can choose other representations, for example, sums of double layer heat potentials. As mentioned above, in this section, we establish the smoothness of the domain-to-energy map for the solution of two boundary value problems (cf. Theorems \ref{thm smooth energy robin} and \ref{thm smooth energy dir} below). The proofs are based on the general result provided by Theorem \ref{thm energy dependence abstract} and on the smooth dependence of $\phi$-pullback of layer heat potentials proven by Dalla Riva and Luzzini \cite{DaLu23}. Consequently, the smoothness of the domain-to-energy map is in fact independent of the representation chosen for the solutions.
\end{remark}

\subsection{Nonlinear mixed Robin-type perturbed problem}
Let $\Omega^o$ and $\Omega^i$ be as in \eqref{introsetconditions}.  We will consider the following assumptions on the data. 
\begin{equation}\label{introfunconditions 2} 
\begin{split}
    &\text{Let } g^o \in C_{0}^{\frac{\alpha}{2}; \alpha}([0,T] \times \partial\Omega^o), \, g^i  \in C^0([0,T]\times \overline{\Omega^o} \times\R) \text{ with } g^i(0,x,0)=0 \text{ for all } x\in \overline{\Omega^o}.
    \\
    &\text{Assume that the map } [0,T] \times \partial\Omega^i \ni (t,x) \mapsto g^i(t,\phi(x), u(t,x)) \text{ belongs to } C^{\frac{\alpha}{2}; \alpha}([0,T] \times \partial\Omega^i) 
    \\
    &\text{for all } u \in  C^{\frac{1+\alpha}{2}; 1+\alpha}([0,T] \times \partial\Omega^i) \text{ and } \phi \in \mathcal{A}^{\Omega^o}_{\partial\Omega^i}.
    \\
    &\text{Assume that } \mathcal{N}_{g^i} \text{ is of class } C^\infty \text{ from } \mathcal{A}^{\Omega^o}_{\partial\Omega^i} \times C^{\frac{1+\alpha}{2}; 1+\alpha}([0,T] \times \partial\Omega^i) \text{ to } C^{\frac{\alpha}{2}; \alpha}([0,T] \times \partial\Omega^i).
\end{split} 
\end{equation}
Here and in the sequel, we denote by $\mathcal{N}_{g^i}$ the ``Nemytskii superposition operator'', defined by
\begin{equation*}
    \mathcal{N}_{g^i}(\phi, u) (t,x) = g^i(t,\phi(x), u(t,x)) \quad \text{for all } (t,x) \in [0,T] \times \partial\Omega^i,
\end{equation*}
for all $\phi \in \mathcal{A}^{\Omega^o}_{\partial\Omega^i}$ and $u$ a function from $[0,T] \times \partial\Omega^i$ to $\R$.
We consider the problem                
\begin{equation}\label{princeqpertu 2}
\begin{cases}
    \partial_t u - \Delta u = 0 & \quad\text{in } ]0,T] \times (\Omega^o \setminus \overline{\Omega^i[\phi]}), 
    \\
    \frac{\partial}{\partial\nu_{\Omega^o}} u(t,x) = g^o(t,x)& \quad \forall (t,x)\in [0,T] \times \partial \Omega^o, 
    \\
    \frac{\partial}{\partial\nu_{\Omega^i[\phi]}} u (t,x) = g^i(t,x,u(t,x)) & \quad \forall (t,x)\in  [0,T] \times \phi(\partial\Omega^i),
    \\
    u(0,\cdot)=0 & \quad \text{in } \overline{\Omega^o} \setminus \Omega^i[\phi],
\end{cases}
\end{equation}
for $\phi \in \mathcal{A}^{\Omega^o}_{\partial\Omega^i}$. Problem \eqref{princeqpertu 2} is the perturbed version of the problem in the fixed annulus $\Omega^o \setminus \overline{\Omega^i}$, namely,
\begin{equation}\label{princeq 2}
\begin{cases}
    \partial_t u - \Delta u = 0 & \quad\text{in } ]0,T] \times (\Omega^o \setminus \overline{\Omega^i}), 
    \\
    \frac{\partial}{\partial\nu_{\Omega^o}} u(t,x) = g^o(t,x)& \quad \forall (t,x)\in [0,T] \times \partial \Omega^o, 
    \\
    \frac{\partial}{\partial\nu_{\Omega^i}} u (t,x) = g^i(t,x,u(t,x)) & \quad \forall (t,x)\in  [0,T] \times \partial \Omega^i,
    \\
    u(0,\cdot)=0 & \quad \text{in } \overline{\Omega^o} \setminus \Omega^i.
    \end{cases}
\end{equation}
To shorten our notation, we define 
\[
\mathcal{X}_0^{\frac{\alpha}{2}; \alpha} :=  C_0^{\frac{\alpha}{2}; \alpha}([0,T] \times \partial\Omega^o) \times C_0^{\frac{\alpha}{2}; \alpha}([0,T] \times \partial \Omega^i).
\]
Let $\mathcal{M}:=(\mathcal{M}_1,\mathcal{M}_2)$ be the map from $\mathcal{A}^{\Omega^o}_{\partial\Omega^i}\times\mathcal{X}_0^{\frac\alpha2;\alpha}$ to $\mathcal{X}_0^{\frac\alpha2;\alpha}$ defined by
\begin{equation}\label{eq M Robin}
\begin{aligned}
\mathcal{M}_1[\phi,\xi,\gamma]
&:=\left(\frac12I+W_{\partial\Omega^o}^*\right)[\xi]
+\nu_{\Omega^o}\cdot\nabla v^-_{\Omega^i[\phi]}
\left[\gamma\circ(\phi^T)^{(-1)}\right]_{|[0,T]\times\partial\Omega^o}
-g^o
&&\text{on }[0,T]\times\partial\Omega^o,
\\
\mathcal{M}_2[\phi,\xi,\gamma]
&:=\left(-\frac12I+W_{\phi(\partial\Omega^i)}^*\right)
\left[\gamma\circ(\phi^T)^{(-1)}\right]\circ\phi^T
\\
&\quad +(\nu_{\Omega^i[\phi]}\circ\phi)\cdot
\left(\nabla v^+_{\Omega^o}[\xi]_{|[0,T]\times\phi(\partial\Omega^i)}\circ\phi^T\right)
\\
&\quad -\mathcal{N}_{g^i} \left( \phi, v^+_{\Omega^o}[\xi]_{|[0,T]\times\phi(\partial\Omega^i)}\circ\phi^T + V_{\phi(\partial\Omega^i)} \left[\gamma\circ(\phi^T)^{(-1)}\right]\circ\phi^T \right)
&&\text{on }[0,T]\times\partial\Omega^i,
\end{aligned}
\end{equation}
for all $(\phi,\xi,\gamma)\in\mathcal{A}^{\Omega^o}_{\partial\Omega^i}\times\mathcal{X}_0^{\frac\alpha2;\alpha}$.

In the following theorem, we summarize the existence result for the perturbed problem \eqref{princeqpertu 2} under assumption \eqref{introfunconditions 2} from \cite[Sect.\,3]{DaLuMoMu24.2}, where the analysis has been carried out. We mention that, as done in \cite{DaLuMoMu24.2}, we will assume the existence of a solution 
\[
\mathbf{v}_0 \in C_{0}^{\frac{1+\alpha}{2}; 1+\alpha}([0,T] \times (\overline{\Omega^o} \setminus \Omega^i))
\]
of problem \eqref{princeq 2}. For sufficient conditions on $g^i$ ensuring the existence of such solution $\mathbf{v}_0$, we refer to the argument in \cite{DaMoMu19} (see also \cite{Mo24} and \cite{Friedman2008}). We just mention that in \cite{DaLuMoMu24.2} the proofs of points (ii) and (iii) below are based on Theorem \ref{thm dependence V phi W phi} (ii) and assumption \eqref{introfunconditions 2}.

\begin{theorem}\label{sec 4: thm recall}
Let $\mathcal{M}$ be given by \eqref{eq M Robin}. Then the following holds.
\begin{itemize}
    \item[(i)] Let $(\phi, \xi,\gamma) \in \mathcal{A}^{\Omega^o}_{\partial\Omega^i} \times \mathcal{X}_0^{\frac{\alpha}{2}; \alpha}$. Then the function
    \begin{equation*}
        \mathbf{v}_{\Omega^o,\Omega^i[\phi]}\left[\xi,\gamma \circ (\phi^T)^{(-1)} \right] := \left( v^+_{\Omega^o} [\xi] + v^-_{\Omega^i[\phi]} \left[\gamma\circ (\phi^T)^{(-1)}\right] \right)_{| [0,T] \times (\overline{\Omega^o} \setminus {\Omega^i[\phi]})}
    \end{equation*}
    (cf.\, \eqref{eq: representation single layer}) is a solution in $C_{0}^{\frac{1+\alpha}{2}; 1+\alpha}([0,T] \times (\overline{\Omega^o} \setminus \Omega^i[\phi]))$ of problem \eqref{princeqpertu 2} if and only if 
    \begin{equation*}
        \mathcal{M}[\phi,\xi,\gamma] = (0,0).
    \end{equation*}
    In particular, there exists a unique pair $(\xi_0,\gamma_0) \in \mathcal{X}_0^{\frac{\alpha}{2}; \alpha}$ such that 
    \[
    \mathbf{v}_0 =  \mathbf{v}_{\Omega^o,\Omega^i[\phi_0]}\left[\xi_0,\gamma_0 \circ (\phi_0^T)^{(-1)} \right] = \left( v^+_{\Omega^o} [\xi_0] + v^-_{\Omega^i} [\gamma_0] \right)_{| [0,T] \times (\overline{\Omega^o} \setminus {\Omega^i})}
    \]
    is the unique solution in $C_{0}^{\frac{1+\alpha}{2}; 1+\alpha}([0,T] \times (\overline{\Omega^o} \setminus \Omega^i))$ of problem \eqref{princeq 2} and 
    \[
    \mathcal{M}[\phi_0,\xi_0,\gamma_0] = (0,0).
    \]

    \item[(ii)] The map $\mathcal{M}$ from $\mathcal{A}^{\Omega^o}_{\partial\Omega^i} \times \mathcal{X}_0^{\frac{\alpha}{2}; \alpha}$ to $\mathcal{X}_0^{\frac{\alpha}{2}; \alpha}$ is of class $C^{\infty}$.

    \item[(iii)] The partial differential of $\mathcal{M}$ with respect to $(\xi,\gamma)$ evaluated at the point $(\phi_0,\xi_0,\gamma_0)$, which we denote by $\partial_{(\xi,\gamma)}\mathcal{M}[\phi_0,\xi_0,\gamma_0]$ is a homeomorphism from $\mathcal{X}_0^{\frac{\alpha}{2}; \alpha}$ into itself.

    \item[(iv)] There exist two open neighborhoods $\mathcal{Q}_0$ of $\phi_0$ in $\mathcal{A}^{\Omega^o}_{\partial\Omega^i}$ and $\mathcal{H}_0$ of $(\xi_0,\gamma_0)$ in $\mathcal{X}_0^{\frac{\alpha}{2}; \alpha}$ and a $C^{\infty}$ map 
    \[
    \Xi := (\Xi_1,\Xi_2): \mathcal{Q}_0 \to \mathcal{H}_0
    \]
    such that the set of zeros of $\mathcal{M}$ in $\mathcal{Q}_0 \times \mathcal{H}_0$ coincides with the graph of the function $\Xi$. In particular,
    \begin{equation*}
    \mathcal{M}[\phi, \Xi_1[\phi],\Xi_2[\phi]] = 0  \quad \forall\phi \in \mathcal{Q}_0, \qquad  \Xi[\phi_0]= (\Xi_1[\phi_0],\Xi_2[\phi_0])= (\xi_0,\gamma_0).
    \end{equation*}

    \item[(v)] For each $\phi \in \mathcal{Q}_0$, the function $\mathbf{v}_\phi \in C_{0}^{\frac{1+\alpha}{2}; 1+\alpha}([0,T] \times (\overline{\Omega^o} \setminus \Omega^i[\phi] ))$ given by
    \begin{equation*}
    \begin{aligned}
        \mathbf{v}_{\phi} :=& \,\mathbf{v}_{\Omega^o,\Omega^i[\phi]}\left[\Xi_1[\phi],\Xi_2[\phi]\circ (\phi^T)^{(-1)}\right]
        \\
        =&\left( v^+_{\Omega^o} [\Xi_1[\phi]] + v^-_{\Omega^i[\phi]}
        \left[\Xi_2[\phi]\circ (\phi^T)^{(-1)}\right] \right)_{| [0,T] \times (\overline{\Omega^o} \setminus {\Omega^i[\phi]})}
    \end{aligned}
    \end{equation*}
    is a solution of \eqref{princeqpertu 2} and $\mathbf{v}_{\phi_0}= \mathbf{v}_0$.
    \end{itemize}
\end{theorem}

We now recall that suitable restriction of the family of functions $\mathbf{v}_\phi$ depends smoothly on $\phi$. We provide the following adaptation of \cite[Thm.\,5]{DaLuMoMu24.2}, where we consider a domain $\Omega_{\mathtt{int}}$ contained in $\overline{\Omega^o} \setminus \overline{\Omega^i}$, possibly touching only the external boundary $\partial\Omega^o$ (see Figure \ref{fig: perturbed domains and omegaint}). Although we allow $\Omega_\mathtt{int} \cap \partial\Omega^o \neq \emptyset$, the proof is essentially the same, due to the linear continuity of the map $v^+_{\Omega^o}[\cdot]$ up to the boundary provided by Theorem \ref{thmsl}\,(iv) and the regularity of the map $\Xi_1[\cdot]$ provided by Theorem \ref{sec 4: thm recall}. 

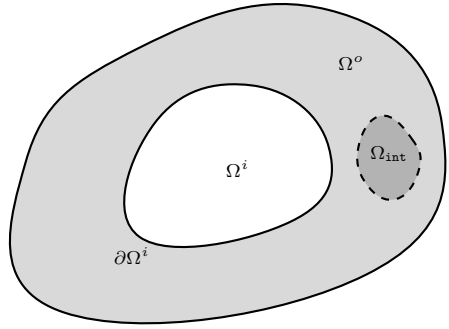
\begin{figure}[ht!]
    \centering
    \begin{tikzpicture}

    \fill[gray!30]
      plot[smooth cycle, tension=0.8]
      coordinates {(0,0) (4,0.5) (5,2.5) (3.5,4) (1,3.5) (-0.5,2)};
    \draw[thick, black]
      plot[smooth cycle, tension=0.8]
      coordinates {(0,0) (4,0.5) (5,2.5) (3.5,4) (1,3.5) (-0.5,2)};
    
    \def\inner{
      (1,1) (3,1.2) (3.5,2.2) (2.5,3) (1.2,2.5)
    }
    
    \fill[white]
      plot[smooth cycle, tension=0.8]
      coordinates \inner;
    
    \draw[thick, black]
      plot[smooth cycle, tension=0.8]
      coordinates \inner;

    \def\omegaint{
    (4.2,2.6) (4.6,2.3) (4.7,1.9) (4.4,1.5) (4.0,1.7) (3.9,2.2)
    }

    \fill[gray!60]
      plot[smooth cycle, tension=0.8]
      coordinates \omegaint;

    \draw[thick, dashed]
      plot[smooth cycle, tension=0.8]
      coordinates \omegaint;
    
    \node at (3.8,3.3) {\tiny$\Omega^o$};
    \node at (2.3,1.9) {\tiny$\Omega^i$};
    \node at (0.9,0.75) {\tiny$\partial\Omega^i$};
    \node at (4.3,2.1) {\tiny$\Omega_{\mathtt{int}}$};
    
    \end{tikzpicture}
    \caption{The domains $\Omega^o$, $\Omega^i$, and $\Omega_{\mathtt{int}}$.}
    \label{fig: perturbed domains and omegaint}
\end{figure}

\begin{theorem}\label{sec 4: thm smootheness}
    Let $\mathcal{Q}_0$, $\Xi = (\Xi_1,\Xi_2)$ and $\mathbf{v}_\phi$ be as in Theorem \ref{sec 4: thm recall}. Let $\Omega_\mathtt{int}$ be a bounded open subset of $\Omega^o\setminus \overline{\Omega^i}$ of class $C^{1,\alpha}$ such that
    \[
    \overline{\Omega_\mathtt{int}} \subseteq  \overline{\Omega^o} \setminus \overline{\Omega^i}.
    \]
    Let $\mathcal{Q}_\mathtt{int} \subseteq \mathcal{Q}_0$ be an open neighborhood of $\phi_0$ such that  
    \begin{equation*}
        \overline{\Omega_\mathtt{int}} \subseteq \overline{\Omega^o} \setminus \overline{\Omega^i[\phi]} \quad \text{for all }\phi \in \mathcal{Q}_\mathtt{int}.
    \end{equation*}
    Then the map from $\mathcal{Q}_\mathtt{int}$ to $C_{0}^{\frac{1+\alpha}{2}; 1+\alpha}([0,T] \times \overline{\Omega_\mathtt{int}})$ that takes $\phi$ to $\left(\mathbf{v}_{\phi}\right)_{|[0,T]\times\overline{\Omega_\mathtt{int}}}$ is of class $C^\infty$.
\end{theorem}

We are now ready to state and prove the smoothness upon the perturbation parameter of the energy for the solution of the nonlinear mixed Robin-type problem \eqref{princeqpertu 2}.

\begin{theorem}\label{thm smooth energy robin}
    Let $\mathcal{Q}_0$, $\Xi=(\Xi_1,\Xi_2)$ and $\mathbf{v}_\phi$ be as in Theorem \ref{sec 4: thm recall}. Then there exists an open neighborhood $\widehat{\mathcal W}_0$ of $\phi_0$ in $\mathcal A_{\partial\Omega^i}^{\Omega^o}$ such that the map from $\widehat{\mathcal{W}}_0$ to $C^1([0,T])$, which takes $\phi$ to $\mathbf{e}_\phi$ given by
    \begin{equation*}
    \mathbf{e}_\phi(t) := \frac{1}{2} \int_{\Omega^o \setminus \overline{\Omega^i[\phi]}} (\mathbf{v}_\phi(t,y))^2 \d y \quad \forall t \in [0,T]
    \end{equation*}
    is of class $C^\infty$. In particular, for every $t\in(0,T)$,
    \begin{equation}\label{eq: derivative of e 2}
    \begin{aligned}
        \frac{\d{}}{\d t}\mathbf{e}_\phi(t)
        &=-\int_{\Omega^o\setminus\overline{\Omega^i[\phi]}}|\nabla \mathbf{v}_\phi(t,y)|^2 \d y
        +\int_{\partial\Omega^o} \mathbf{v}_\phi(t,y)g^o(t,y) \d\sigma_y
        \\
        &\quad -\int_{\partial\Omega^i} \mathbf{v}_\phi(t,\phi(y)) \, g^i(t,\phi(y),\mathbf{v}_\phi(t,\phi(y))) \, \tilde\sigma_n[\phi](y) \d\sigma_y,
    \end{aligned}
    \end{equation}
    and the right-hand side extends continuously to $[0,T]$.
\end{theorem}

\begin{proof}
    The result follows by Theorem \ref{thm energy dependence abstract}, once we have verified the required assumptions. The family of functions $\{\mathbf{v}_\phi\}_{\phi \in \mathcal{Q}_0}$ provided by Theorem \ref{sec 4: thm recall} satisfies assumption (i) of Theorem \ref{thm energy dependence abstract}. Hence, we just verify the existence of a neighborhood $\widehat{\mathcal W}_0$ of $\phi_0$ in $\mathcal A_{\partial\Omega^i}^{\Omega^o}$ such that for every $\phi\in \widehat{\mathcal{W}}_0$ the pulled back solution on the fixed domain $\widehat\Omega$ given by
    \begin{equation}\label{eq: V_phi v_phi Psi phi}
    \mathbf{V}_\phi(t,x):= \left( \mathbf{v}_\phi \circ \Psi[\phi]^T \right) (t,x) 
    \quad \text{with }
    \Psi[\phi]^T(t,x):=(t,\Psi[\phi](x)) \text{ for } (t,x) \in \widehat\Omega
    \end{equation}  
    satisfies assumption (ii) of Theorem \ref{thm energy dependence abstract}.

    To this aim, we just decompose the domain $\widehat\Omega$ into three regions: the near, the intermediate and the far from the cavity region. 
    
    \smallskip
    \textit{Step 1: near region}. Let $\Omega_{\mathtt{near}}:= \Omega^{i,-}_{\omega,\rho_1}$. By Lemma \ref{lemma diffeo}\,(ii) (cf.\,\eqref{eq Psi phi 2}), we know that $\Psi[\phi] = \mathbf{E}[\phi]$ on $\overline{\Omega_{\mathtt{near}}}$ for every $\phi \in \mathcal{W}^\ast_0$. Hence, by Theorem \ref{sec 4: thm recall}\,(v) along with \eqref{eq: V_phi v_phi Psi phi}, we deduce that
    \begin{equation*}
    \begin{split}
    \mathbf{V}_\phi(t,x) &=  v^+_{\Omega^o} [\Xi_1[\phi]](t, \mathbf{E}[\phi](x)) + v^-_{\Omega^i[\phi]}
    \left[\Xi_2[\phi]\circ (\phi^T)^{(-1)}\right](t, \mathbf{E}[\phi](x))
    \\
    & = v^+_{\Omega^o} [\Xi_1[\phi]](t, \mathbf{E}[\phi](x)) + V^-_{\mathbf{E}[\phi], \partial \Omega^i}
    [\Xi_2[\phi]](t,x) \quad \forall (t,x) \in [0,T] \times \overline{\Omega_{\mathtt{near}}},
    \end{split}
    \end{equation*}
    for every $\phi \in \widehat{\mathcal{W}}_0 := \mathcal{W}^\ast_0 \cap \mathcal{Q}_0$. The map that takes $\phi$ to the function of the variables $(t,x) \in [0,T] \times \overline{\Omega_{\mathtt{near}}}$ given by
    \[
    v^+_{\Omega^o} [\Xi_1[\phi]](t, \mathbf{E}[\phi](x)) = \int_{0}^{t} \int_{\partial \Omega^o} S_n(t-\tau, \mathbf{E}[\phi](x)-y) \, \Xi_1[\phi](\tau, y) \d \sigma_y \d \tau 
    \]
    is of class $C^\infty$ from $\widehat{\mathcal{W}}_0$ to $C_0^{\frac{1+\alpha}{2}; 1+\alpha}([0,T] \times \overline{\Omega_{\mathtt{near}}})$, by Lemma \ref{lemma extension operator}\,(ii), by \cite[Lems.\,A.1,\,A.2\,\&\,A.3]{DaLu23} and by Theorem \ref{sec 4: thm recall}\,(iv) on the regularity of the map $\Xi_1[\cdot]$.

    Moreover, by Lemma \ref{lemma extension operator}\,(ii), Theorem \ref{sec 4: thm recall}\,(iv) and Theorem \ref{thm dependence V Phi W Phi}, we deduce that the map from $\widehat{\mathcal{W}}_0$ to $ C_0^{\frac{1+\alpha}{2}; 1+\alpha}([0, T] \times \overline{\Omega_{\mathtt{near}}})$ that takes $\phi$ to the function $V^-_{\mathbf{E}[\phi], \partial \Omega^i}[\Xi_2[\phi]]$ is of class $C^\infty$.

    Hence, the map from $\widehat{\mathcal{W}}_0$ to $ C_0^{\frac{1+\alpha}{2}; 1+\alpha}([0, T] \times \overline{\Omega_{\mathtt{near}}})$ that takes $\phi$ to the function $(\mathbf{V}_\phi)_{|[0, T] \times \overline{\Omega_{\mathtt{near}}}}$ is of class $C^{\infty}$.

    \smallskip
    \textit{Step 2: intermediate region}. Consider now the domain $\Omega_{\mathtt{med}}:= \Omega^{i,-}_{\omega,\rho_2} \setminus \overline{\Omega^{i,-}_{\omega,\rho_1}}$. The set $\Omega_{\mathtt{med}}$ is strictly contained in $\widehat\Omega$ and has positive distance from both $\partial\Omega^o$ and $\partial\Omega^i$. By Lemma \ref{lemma diffeo}, we know that the map $\Psi$ is of class $C^\infty$ from $\mathcal{W}^\ast_0$ to $C^{1,\alpha}(\overline{\Omega^o},\R^n)$ and $\Psi[\phi_0] = \mathrm{id}_{\overline{\Omega^o}}$. Thus, there exists an open neighborhood $\widehat{\mathcal{W}}_0 \subseteq \mathcal{W}^\ast_0 \cap \mathcal{Q}_0$, such that 
    \begin{equation}\label{eq: cond omega med}
    \overline{\Psi[\phi](\Omega_{\mathtt{med}})} \subset \Omega^o \setminus \overline{\Omega^i[\phi]} \quad \text{for all } \phi \in \widehat{\mathcal{W}}_0,
    \end{equation}
    i.e. $\Psi[\phi](\Omega_{\mathtt{med}})$ has positive distance both from $\partial\Omega^o$ and $\phi(\partial\Omega^i)$ for every $\phi \in \widehat{\mathcal{W}}_0$. Assume \eqref{eq: cond omega med}. Then, by Theorem \ref{sec 4: thm recall}\,(v) and \eqref{eq: V_phi v_phi Psi phi}, for all $(t,x) \in [0,T]\times \Omega_{\mathtt{med}}$ we have 
    \begin{equation}\label{eq: V phi on omega med}
    \begin{split}
    \mathbf{V}_\phi(t,x) = & \int_0^t \int_{\partial\Omega^o} S_n (t-\tau,\Psi[\phi](x) - y) \,\Xi_1[\phi](\tau,y) \d \sigma_y \d \tau
    \\
    & + \int_0^t \int_{\partial\Omega^i} S_n (t-\tau,\Psi[\phi](x) - \phi(y) ) \, \Xi_2[\phi](\tau,y)\, \tilde{\sigma}_n[\phi](y) \d \sigma_y \d \tau.
    \end{split}
    \end{equation}

    Now, we observe that, by the property \eqref{eq: cond omega med} and by Lemma \ref{lemma diffeo}, the maps that associate a diffeomorphism $\phi$ to the functions
    \begin{equation*}
        \overline{\Omega_\mathtt{med}} \times \partial\Omega^o \ni (x,y) \longmapsto \Psi[\phi](x) - y \in \R^n \setminus\{0\}
    \end{equation*}
    and
    \begin{equation*}
        \overline{\Omega_\mathtt{med}} \times \partial\Omega^i \ni (x,y) \longmapsto \Psi[\phi](x) -\phi(y) \in \R^n \setminus\{0\}
    \end{equation*}
    are both of class $C^{\infty}$ from $\widehat{\mathcal{W}}_0$ to $C^{1,\alpha}(\overline{\Omega_\mathtt{med}} \times \partial\Omega^o, \R^n \setminus\{0\})$ and to $C^{1,\alpha}(\overline{\Omega_\mathtt{med}} \times \partial\Omega^i, \R^n \setminus\{0\})$, respectively. 
    Then, by \cite[Lems.\,A.1\,\&\,A.3]{DaLu23} regarding the regularity of the superposition operators, we deduce that the maps that take a $\phi$ to the functions
    \begin{equation*}
        [0,T] \times \overline{\Omega_\mathtt{med}} \times \partial\Omega^o \ni (t,x,y) \longmapsto S_n (t,\Psi[\phi](x) - y) \in \R 
    \end{equation*}
    and
    \begin{equation*}
        [0,T] \times \overline{\Omega_\mathtt{med}} \times \partial\Omega^i \ni (t,x,y) \longmapsto S_n (t,\Psi[\phi](x) - \phi(y)) \in \R 
    \end{equation*}
    are of class $C^{\infty}$ from $\widehat{\mathcal{W}}_0$ to $C_0^{\frac{1+\alpha}{2}; 1+\alpha}([0,T] \times (\overline{\Omega_\mathtt{med}} \times \partial\Omega^o))$ and to $C_0^{\frac{1+\alpha}{2}; 1+\alpha}([0,T] \times (\overline{\Omega_\mathtt{med}} \times \partial\Omega^i))$, respectively.

    Finally, by \cite[Lem.\,A.2]{DaLu23}, by the regularity of the maps $\Xi_1[\cdot]$ and $\Xi_2[\cdot]$ provided by Theorem \ref{sec 4: thm recall}\,(iv), and by Lemma \ref{lemma change of variable}, we deduce that the right-hand side of \eqref{eq: V phi on omega med} (i.e., $(\mathbf{V}_\phi)_{|[0, T] \times \overline{\Omega_{\mathtt{med}}}}$) defines a $C^\infty$ map from $\widehat{\mathcal{W}}_0$ to $C_0^{\frac{1+\alpha}{2}; 1+\alpha}([0,T] \times \overline{\Omega_\mathtt{med}})$.

    \smallskip
    \textit{Step 3: far region}. Consider finally the domain $\Omega_{\mathtt{far}} := \widehat\Omega \setminus\overline{\Omega^{i,-}_{\omega,\rho_2}}$. Since $\Psi[\phi] = \mathrm{id}_{\overline{\Omega^o}}$ on $\overline{\Omega_{\mathtt{far}}}$ (cf.\,\eqref{eq Psi phi 2}), we have $\mathbf{V}_\phi = \mathbf{v}_\phi$ on $\overline{\Omega_{\mathtt{far}}}$ (see \eqref{eq: V_phi v_phi Psi phi}). Moreover, we have that $\overline{\Omega_{\mathtt{far}}} \subseteq \overline{\Omega^o}\setminus\overline{\Omega^i}$, hence there exists an open neighborhood $\widehat{\mathcal{W}}_0 \subseteq \mathcal{W}^\ast_0 \cap \mathcal{Q}_0$ such that 
    \[
    \overline{\Omega_{\mathtt{far}}} \subseteq \overline{\Omega^o} \setminus \overline{\Omega^i[\phi]}  \quad \text{for all } \phi \in \widehat{\mathcal{W}}_0.
    \]
    Hence, by Theorem \ref{sec 4: thm smootheness} with $\Omega_{\mathtt{int}} = \Omega_{\mathtt{far}}$ we know that the map from $\widehat{\mathcal{W}}_0$ to $C_{0}^{\frac{1+\alpha}{2}; 1+\alpha}([0,T] \times \overline{\Omega_\mathtt{far}})$ that takes $\phi$ to $(\mathbf{V}_\phi)_{|[0,T] \times \overline{\Omega_\mathtt{far}}}$ is of class $C^\infty$.

    \smallskip
    After shrinking $\widehat{\mathcal{W}}_0$ once more if necessary, we may assume that all three preceding constructions hold on the same neighborhood. The near, intermediate, and far regions form a finite $C^{1,\alpha}$ decomposition of $\overline{\widehat\Omega}$, and the three functions considered above are restrictions of the same pull-back $\mathbf{V}_\phi$. Hence their traces and first spatial derivatives agree on the common interfaces. By the standard finite-patching estimate for parabolic Schauder norms, the map $\phi \mapsto \mathbf{V}_\phi$ is therefore of class $C^\infty$ from $\widehat{\mathcal{W}}_0$ to $C_0^{\frac{1+\alpha}{2};1+\alpha}([0,T]\times\overline{\widehat\Omega})$.

    Hence, the family of functions $\{\mathbf{v}_\phi\}_{\phi \in \widehat{\mathcal{W}}_0}$ satisfies the assumptions of Theorem \ref{thm energy dependence abstract}. In particular, \eqref{eq: derivative of e 2} follows from \eqref{eq: derivative of e general} using the boundary conditions in \eqref{princeqpertu 2}. Thus the proof is complete. 
    
\end{proof}

\subsection{Linear Dirichlet perturbed problem}
Let $\Omega^o$ and $\Omega^i$ be as in \eqref{introsetconditions}. We fix two functions
\begin{equation}\label{introfunconditions} 
f^o \in C_{0}^{\frac{1+\alpha}{2}; 1+\alpha}([0,T] \times \partial\Omega^o), \, f^i \in C_{0}^{\frac{1+\alpha}{2}; 1+\alpha}([0,T] \times \partial\Omega^i).
\end{equation}
For $\phi \in \mathcal{A}^{\Omega^o}_{\partial\Omega^i}$, we consider the following problem
\begin{equation}\label{princeqpertu}
\begin{cases}
    \partial_t u - \Delta u = 0 & \quad\text{in } ]0,T] \times (\Omega^o \setminus \overline{\Omega^i[\phi]}), 
    \\
    u(t,x) = f^o(t,x)& \quad \forall (t,x)\in [0,T] \times \partial \Omega^o, 
    \\
    u(t,x) = f^i(t,\phi^{-1}(x)) & \quad \forall (t,x)\in  [0,T] \times \phi(\partial\Omega^i),
    \\
    u(0,\cdot)=0 & \quad \text{in } \overline{\Omega^o} \setminus \Omega^i[\phi],
    \end{cases}
\end{equation}
which is the perturbed version of 
\begin{equation*}
    \begin{cases}
    \partial_t u - \Delta u = 0 & \quad\text{in } ]0,T] \times (\Omega^o \setminus \overline{\Omega^i}), 
    \\
    u (t,x) = f^o(t,x) & \quad \forall (t,x) \in [0,T] \times \partial \Omega^o,
    \\
    u (t,x) = f^i(t,x) & \quad \forall (t,x) \in [0,T] \times \partial \Omega^i,
    \\
    u (0,\cdot)=0 & \quad \text{in } \overline{\Omega^o} \setminus \Omega^i.
    \end{cases}
\end{equation*}
To shorten the notation, set
\[
\mathcal{X}_0^{\frac{1+\alpha}{2};1+\alpha}:= C_0^{\frac{1+\alpha}{2};1+\alpha}([0,T]\times\partial\Omega^o) \times C_0^{\frac{1+\alpha}{2};1+\alpha}([0,T]\times\partial\Omega^i).
\]
Let $\mathcal{M} := (\mathcal{M}_1,\mathcal{M}_2)$ be the map from $\mathcal{A}^{\Omega^o}_{\partial\Omega^i}$ to $\mathcal{L} \left(\mathcal{X}_0^{\frac{\alpha}{2}; \alpha}, \mathcal{X}_0^{\frac{1+\alpha}{2}; 1+\alpha} \right)$ defined by
\begin{equation}\label{eq M Dirichlet}
\begin{aligned}
\mathcal{M}_1[\phi](\mu,\eta) & := V_{\partial\Omega^o}[\mu] + v^-_{\Omega^i[\phi]}
\left[\eta \circ (\phi^T)^{(-1)}\right]_{|[0,T] \times \partial\Omega^o} && \text{on } [0,T] \times \partial \Omega^o,
\\
\mathcal{M}_2[\phi](\mu,\eta) & :=
V_{\phi(\partial\Omega^i)}\left[\eta \circ (\phi^T)^{(-1)}\right] \circ \phi^T + v^+_{\Omega^o}[\mu]_{|[0,T] \times \phi(\partial\Omega^i)}\circ \phi^T &&\text{on } [0,T] \times \partial \Omega^i,
\end{aligned}
\end{equation}
for all $(\phi,\mu, \eta) \in \mathcal{A}^{\Omega^o}_{\partial\Omega^i} \times \mathcal{X}_0^{\frac{\alpha}{2}; \alpha}$. We now recall the following results provided in \cite{DaLuMu25}, assuming \eqref{introfunconditions}.
In particular, the results recalled here are simplified, since we are just interested in the dependence of the solution on the perturbation parameter $\phi$ (and not also on the Dirichlet data, as in \cite{DaLuMu25}).
We just mention that Theorem \ref{thm dependence V phi W phi} (i) is central in \cite{DaLuMu25} for the proofs of points (ii) and (iii).

\begin{theorem}\label{sec 5: thm recall dir}
Let $\mathcal{M}$ be given by \eqref{eq M Dirichlet}. Then the following holds.
\begin{itemize}
    \item[(i)] Let $(\phi,\mu, \eta) \in \mathcal{A}^{\Omega^o}_{\partial\Omega^i} \times \mathcal{X}_0^{\frac{\alpha}{2}; \alpha}$. Then, the function 
    \[
    u_{\Omega^o,\Omega^i[\phi]}\left[\mu,\eta\circ (\phi^T)^{(-1)}\right] := \left( v^+_{\Omega^o} [\mu] + v^-_{\Omega^i[\phi]} \left[\eta\circ (\phi^T)^{(-1)}\right] \right)_{| [0,T] \times (\overline{\Omega^o} \setminus {\Omega^i[\phi]})} 
    \]
    (cf. \eqref{eq: representation single layer}) is a solution in $C_{0}^{\frac{1+\alpha}{2}; 1+\alpha}([0,T] \times (\overline{\Omega^o} \setminus \Omega^i[\phi]))$ of problem \eqref{princeqpertu} if and only if 
    \begin{equation*}
    \mathcal{M}[\phi](\mu, \eta) = (f^o,f^i).
    \end{equation*}
    Moreover, for every $\phi \in \mathcal{A}^{\Omega^o}_{\partial\Omega^i}$, there exists a unique pair $(\Lambda_1[\phi],\Lambda_2[\phi]) \in \mathcal{X}_0^{\frac{\alpha}{2}; \alpha}$ such that 
    the function $\mathbf{u}_\phi \in C_{0}^{\frac{1+\alpha}{2}; 1+\alpha}([0,T] \times (\overline{\Omega^o} \setminus \Omega^i[\phi] ))$ given by
    \begin{equation*}
    \begin{split}
        \mathbf{u}_{\phi} :=& \,u_{\Omega^o,\Omega^i[\phi]}\left[\Lambda_1[\phi],\Lambda_2[\phi]\circ (\phi^T)^{(-1)}\right] 
        \\
        =&\left( v^+_{\Omega^o} [\Lambda_1[\phi]] + v^-_{\Omega^i[\phi]}
        \left[\Lambda_2[\phi]\circ (\phi^T)^{(-1)}\right] \right)_{| [0,T] \times (\overline{\Omega^o} \setminus {\Omega^i[\phi]})}
    \end{split}
    \end{equation*}
    is a solution of \eqref{princeqpertu}.

    \item[(ii)] The map $\mathcal{M}$ from $\mathcal{A}^{\Omega^o}_{\partial\Omega^i}$ to $\mathcal{L} \left( \mathcal{X}_0^{\frac{\alpha}{2}; \alpha}, \mathcal{X}_0^{\frac{1+\alpha}{2}; 1+\alpha} \right)$ is of class $C^{\infty}$. In particular, for every $\phi \in \mathcal{A}^{\Omega^o}_{\partial\Omega^i}$, $\mathcal{M}[\phi]$ is linear, continuous and invertible.

    \item[(iii)] The map $\Lambda$ from $\mathcal{A}^{\Omega^o}_{\partial\Omega^i}$ to $\mathcal{X}_0^{\frac{\alpha}{2}; \alpha}$ which takes $\phi$ to $\Lambda[\phi] = (\Lambda_1[\phi],\Lambda_2[\phi])$ is of class $C^{\infty}$.
    \end{itemize}
\end{theorem}

Then, analogously to Theorem \ref{sec 4: thm smootheness}, we have the following adaptation of \cite[Thm. 3.7]{DaLuMu25}.

\begin{theorem}\label{sec 5: thm smootheness dir}
    Let $\mathbf{u}_\phi$ be as in Theorem \ref{sec 5: thm recall dir}. Let $\Omega_\mathtt{int}$ be a bounded open subset of $\Omega^o\setminus \overline{\Omega^i}$ of class $C^{1,\alpha}$ such that
    \[
    \overline{\Omega_\mathtt{int}} \subseteq  \overline{\Omega^o} \setminus \overline{\Omega^i}.
    \]
    Let $\mathcal{Q}_\mathtt{int} \subseteq \mathcal{A}^{\Omega^o}_{\partial\Omega^i}$ be an open neighborhood of $\phi_0$ such that  
    \begin{equation*}
        \overline{\Omega_\mathtt{int}} \subseteq \overline{\Omega^o} \setminus \overline{\Omega^i[\phi]} \quad \text{for all }\phi \in \mathcal{Q}_\mathtt{int}.
    \end{equation*}
    Then, the map from $\mathcal{Q}_\mathtt{int}$ to $C_{0}^{\frac{1+\alpha}{2}; 1+\alpha}([0,T] \times \overline{\Omega_\mathtt{int}})$ that takes $\phi$ to $\left(\mathbf{u}_{\phi}\right)_{|[0,T]\times\overline{\Omega_\mathtt{int}}}$ is of class $C^\infty$.
\end{theorem}

Thus, we are now ready to state the smoothness upon the perturbation parameter of the energy for the solution of the linear Dirichlet perturbed problem \eqref{princeqpertu}.

\begin{theorem}\label{thm smooth energy dir}
    Let $\mathbf{u}_\phi$ be as in Theorem \ref{sec 5: thm recall dir}. Then there exists an open neighborhood $\widehat{\mathcal W}_0$ of $\phi_0$ in $\mathcal A_{\partial\Omega^i}^{\Omega^o}$ such that the map from $\widehat{\mathcal{W}}_0$ to $C^1([0,T])$, which takes $\phi$ to $\mathbf{e}_\phi$ given by
    \begin{equation*}
    \mathbf{e}_\phi(t) := \frac{1}{2} \int_{\Omega^o \setminus \overline{\Omega^i[\phi]}} (\mathbf{u}_\phi(t,y))^2 \d y \quad \forall t \in [0,T]
    \end{equation*}
    is of class $C^\infty$. In particular, for every $t\in(0,T)$,
    \begin{equation}\label{eq: derivative of e dir}
    \begin{aligned}
        \frac{\d{}}{\d t}\mathbf{e}_\phi(t) &= -\int_{\Omega^o \setminus \overline{\Omega^i[\phi]}} |\nabla \mathbf{u}_\phi(t,y)|^2 \d y + \int_{\partial\Omega^o}
        \nabla \mathbf{u}_\phi(t,y)\cdot \nu_{\Omega^o}(y)\,f^o(t,y) \d\sigma_y
        \\
        &\quad -\int_{\partial\Omega^i} \nabla \mathbf{u}_\phi(t,\phi(y))\cdot \nu_{\Omega^i[\phi]}(\phi(y))\,f^i(t,y) \, \tilde\sigma_n[\phi](y) \d\sigma_y,
    \end{aligned}
    \end{equation}
    and the right-hand side extends continuously to $[0,T]$.
\end{theorem}

\begin{proof}
    In view of the same representation provided by Lemma \ref{lemma rappr} used for both problems \eqref{princeqpertu 2} and \eqref{princeqpertu}, the proof of Theorem \ref{thm smooth energy dir} is analogous to the proof of Theorem \ref{thm smooth energy robin}. In fact, it follows by the general Theorem \ref{thm energy dependence abstract} and by using the results contained in Theorems \ref{sec 5: thm recall dir} and \ref{sec 5: thm smootheness dir}. We just mention that \eqref{eq: derivative of e dir} follows by \eqref{eq: derivative of e general} using the boundary conditions satisfied by $\mathbf{u}_\phi$ (cf. \eqref{princeqpertu}).
\end{proof}

\subsection*{Acknowledgments}

The authors are members of the ``Gruppo Nazionale per l'Analisi Matematica, la Probabilit\`a e le loro Applicazioni'' (GNAMPA) of the ``Istituto Nazionale di Alta Matematica'' (INdAM). R.\,Molinarolo is partially supported by GNAMPA of INdAM. R.\,Molinarolo acknowledges the support by the Project ``Giochi a campo medio, trasporto e ottimizzazione in sistemi auto-organizzati e machine learning'', funded by MUR, D.D. 47/2025, PNRR - Missione 4, Componente 2, Investimento 1.2 - funded by European Union NextGenerationEU, CUP B33C25000380001.
Finally, R.\,Molinarolo wishes to thank P.\,Luzzini and P.\,Musolino for fruitful discussions and kind suggestions that improved the quality of the paper. 

\bibliographystyle{plain}
\bibliography{bibliography}

\end{document}